\documentclass[oneside,reqno,a4paper,11pt]{article}
\usepackage{authblk}
\title{Well-posedness for the incompressible Hall-MHD system with initial magnetic field belonging to $H^{\frac{3}{2}}(\mathbb{R}^3)$}
\author[1]{Shunhang Zhang}

\date{}
\usepackage[sort,space,noadjust]{cite}
\usepackage[colorlinks,
linkcolor=blue,
anchorcolor=blue,
citecolor=blue]{hyperref}
\usepackage{amsfonts}
\usepackage{amsmath,amsfonts,amssymb,amsthm,mathtools}
\usepackage{CJK,bbm,color,soul,supertabular,longtable,verbatim}
\usepackage{bibspacing}
\usepackage[warning, easyold,math]{easyeqn}
\usepackage{graphicx}
\usepackage{color}
\usepackage{enumerate}
\usepackage{titlesec}
\usepackage{breqn}
\newtheorem{thm}{Theorem}[section]
\newtheorem{prop}{Proposition}[section]
\newtheorem{lem}{Lemma}[section]
\newtheorem{rem}{Remark}[section]
\newtheorem{definition}{Definition}[section]

\newcommand\thmref[1]{Theorem~\ref{#1}}
\newcommand\propref[1]{Proposition~\ref{#1}}
\newcommand\lemref[1]{Lemma~\ref{#1}}
\newcommand\remref[1]{Remark~\ref{#1}}

\DeclareMathOperator{\Div}{div}

\DeclareMathOperator{\Id}{Id}
\allowdisplaybreaks[4]

\numberwithin{equation}{section}
\usepackage{breqn}
\begin{document}
\maketitle
\begin{abstract}
In this paper, we first prove the local well-posedness of strong solutions to the incompressible Hall-MHD system for initial data $(u_0,B_0)\in H^{\frac{1}{2}+\sigma}(\mathbb{R}^3)\times H^{\frac{3}{2}}(\mathbb{R}^3)$ with $\sigma\in (0,2)$. In particular, if the viscosity coefficient is equal to the resistivity coefficient, we can reduce $\sigma$ to $0$ with the aid of the new formulation of the Hall-MHD system observed by Danchin and Tan (Commun Partial Differ Equ 46(1):31-65, 2021). Compared with the previous works, our local well-posedness results improve the regularity condition on the initial data. Moreover, we establish the global well-posedness for small initial data in $H^{\frac{1}{2}+\sigma}(\mathbb{R}^3)\times H^{\frac{3}{2}}(\mathbb{R}^3)$ with $\sigma\in (0,2)$,  and get the optimal time-decay rates of solutions.
\end{abstract}
\begin{flushleft}
\end{flushleft}
\begin{flushleft}
\end{flushleft}
\section{Introduction}
This paper is concerned with the Cauchy problem of the viscous and resistive incompressible Hall-magnetohydrodynamics (Hall-MHD) system in $\mathbb{R}^3$:
\begin{equation}\label{main eq}
\begin{cases}
\partial_tu-\mu\Delta u+u\cdot\nabla u+\nabla P=B\cdot\nabla B,\\
\partial_tB-\nu\Delta B-\nabla \times(u\times B)+\kappa\nabla\times((\nabla\times B)\times B)=0,\\
\Div u=\Div B=0,\\
u(x,0)=u_0(x),\quad B(x,0)=B_0(x).
\end{cases}
\end{equation}Here, $u=u(t,x)\in\mathbb{R}^3$, $B(t,x)\in\mathbb{R}^3$ and the scalar function $P=P(t,x)$ with $(t,x)\in \mathbb{R}^+\times\mathbb{R}^3$ stand for the velocity field, the magnetic field and the scalar pressure, respectively. The positive parameters $\mu$, $\nu$ and $\kappa$ are viscosity coefficient, resistivity coefficient and Hall effect coefficient, respectively.\par Notice that if $\Div B_0=0$, the divergence-free property $\Div B=0$ is preserved all the time yielding\begin{equation}
-\nabla \times(u\times B)=u\cdot\nabla B-B\cdot\nabla u.
\end{equation}Compared with the classical MHD system, the Hall-MHD system possesses the so-called Hall term $\kappa\nabla\times((\nabla\times B)\times B)$ in the magnetic equation. Due to the appearance of this extra term, the Hall-MHD system can be applied to describing the magnetic reconnection phenomenon, which occurs in plasmas, star formation, solar flares, neutron stars or geo-dynamo, etc (see \cite{MR2167907,forbes1991magnetic,wardle2004star,balbus2001linear,shalybkov1997hall,mininni2003dynamo}). Meanwhile, since the Hall term contains the second order derivatives, it seems that the mathematical analysis of the Hall-MHD system becomes more complicated than that for the classical MHD system.\par 
Owing to its significance in physics and mathematics, the Hall-MHD sytem has drawn much attention from various researchers. Next, we briefly introduce some of them which are related to our problem. Acheritogaray et al. in \cite{MR2861579} gave a rigorous mathematical derivation of the Hall-MHD system both from a two-fluids model and from a kinetic model. Chae et al. \cite{MR3208454} established the global existence of weak solutions as well as the local well-posedness of smooth solutions for initial data in $H^s(\mathbb{R}^3)$ with $s>\frac{5}{2}$. Moreover, they obtained the global well-posedness of smooth solutions for small initial data. Subsequently, some authors \cite{MR3186849,MR3955610,MR3903774,MR3397315} relaxed the smallness condition by only requiring some suitable Sobolev or Besov norms of initial data to be small. Serrin type blow-up criteria for smooth solutions were derived by Chae and Lee in \cite{MR3186849}. Local well-posedness of strong solutions for initial data with improved regularity condition in $H^s(\mathbb{R}^3)$ were obtained in \cite{MR3513590,MR4060362,MR3770153}, these results relax the restriction on $s$ to $s>\frac{3}{2}$. Recently, Dai \cite{MR4248458} proved the local well-posedness of \eqref{main eq} in $H^{s}(\mathbb{R}^3)\times H^{s+1-\sigma}(\mathbb{R}^3)$ with $s>\frac{1}{2}$ and sufficiently small $\sigma>0$ such that $s+1-\sigma>\frac{3}{2}$. Fan et al. \cite{MR3069946} established the global well-posedness of \eqref{main eq} for a class of axisymmetric initial data. The authors in \cite{MR4193644,MR4188985} proved the global well-posedness of \eqref{main eq} for small initial data in critical Besov spaces by means of some new observations. In addition, it is significant to investigate the large time behavior of solutions to \eqref{main eq}. Chae and Schonbek \cite{MR3097244} obtained the time-decay estimates for weak solutions and smooth solutions with small initial data by using the Fourier splitting techniques. Weng \cite{MR3460237,MR3460222} studied the analyticity and space-time decay estimates for smooth solutions.\par 
On the other hand, observe that the system \eqref{main eq} does not have any scaling invariance as the classical MHD system. However, if we set $B\equiv 0$ in \eqref{main eq}, then the system \eqref{main eq} becomes the incompressible Navier-Stokes system:\begin{equation}
	\begin{cases}
		\partial_tu-\mu\Delta u+u\cdot\nabla u+\nabla P=0,\\
		\Div u=0,\\
		u(x,0)=u_0(x),
	\end{cases}
\end{equation}which is invariant for all $\lambda>0$ under the change
\begin{equation}\label{scaling1}
(u(t,x),P(t,x),u_0(x))\mapsto (\lambda u(\lambda^2t,\lambda x),\lambda^2 P(\lambda^2t,\lambda x),\lambda u_0(\lambda x)).
\end{equation}Also, if $u\equiv 0$, one may see that\begin{equation}
\begin{cases}
	\partial_tB-\nu\Delta B+\kappa\nabla\times((\nabla\times B)\times B)=0,\\
	B(x,0)=B_0(x),
\end{cases}
\end{equation}which is invariant for all $\lambda>0$ under the change
\begin{equation}\label{scaling2}
(B(t,x),B_0(x))\mapsto (B(\lambda^2t,\lambda x),B_0(\lambda x)).
\end{equation}We should point out that all previous results on the local well-posedness of \eqref{main eq} in $H^{s}(\mathbb{R}^3)\times H^{r}(\mathbb{R}^3)$ require $s>\frac{1}{2}$ and $r>\frac{3}{2}$. Whereas, from the viewpoint of scaling invariance, the scalings \eqref{scaling1} and \eqref{scaling2} suggest that $u_0\in H^{\frac{1}{2}}(\mathbb{R}^3)$ or $B_0\in H^{\frac{3}{2}}(\mathbb{R}^3)$ may be a suitable data choice for \eqref{main eq} to establish local well-posedness. Motivated by this, we attempt to do this in paper. The first result of this paper is to prove the local well-posedness of \eqref{main eq} for $B_0\in H^{\frac{3}{2}}(\mathbb{R}^3)$ and $u_0$ is slightly regular than $H^{\frac{1}{2}}(\mathbb{R}^3)$, which is stated as follows.
\begin{thm}\label{thm1}
Let $0<\sigma<2$ and $(u_0,B_0)\in H^{\frac{1}{2}+\sigma}(\mathbb{R}^3)\times H^{\frac{3}{2}}(\mathbb{R}^3)$ with $\Div u_0=\Div B_0=0$. Then there exists a positive time $T$ such that the system \eqref{main eq} admits a unique solution $(u,B)$ with
\begin{equation}\label{local}
\begin{split}
u&\in C([0,T];H^{\frac{1}{2}+\sigma}(\mathbb{R}^3))\cap L^2(0,T;H^{\frac{3}{2}+\sigma}(\mathbb{R}^3)),\\
B&\in C([0,T];H^{\frac{3}{2}}(\mathbb{R}^3))\cap L^2(0,T;H^{\frac{5}{2}}(\mathbb{R}^3)).
\end{split}
\end{equation}
\end{thm}
\begin{rem}\upshape
	Compared with the previous local well-posedness results in nonhomogeneous Sobolev spaces (see e.g. \cite{MR4248458,MR3770153,MR4060362,MR3208454,MR3513590}), our result improves the regularity condition on the initial magnetic field $B_0$ to $B_0\in H^{\frac{3}{2}}(\mathbb{R}^3)$. 
\end{rem}
\begin{rem}\upshape
Let us explain why we can tackle the magnetic field $B$ in $H^{\frac{3}{2}}(\mathbb{R}^3)$. In \cite{MR4248458,MR3770153,MR4060362,MR3208454,MR3513590}, the authors derived the uniform bounds for approximate solutions by establishing the inequality of the following type:\begin{equation}\begin{split}
	&\frac{d}{dt}\Big(\|u(t)\|_{{H}^s}^2+\|B(t)\|_{{H}^r}^2\Big)+\|\nabla u(t)\|_{{H}^{s}}^2+\|\nabla B(t)\|_{{H}^{r}}^2
	\\&\lesssim \Big(\|u(t)\|_{{H}^s}^2+\|B(t)\|_{{H}^r}^2\Big)^{1+\alpha}\quad\text{for some}\quad \alpha>0.
\end{split}\end{equation}However, when tackling the $\dot{H}^r$ estimates of the magnetic equation, one usually use \lemref{commues3} and interpolation to treat the Hall term as\begin{equation}\begin{split}
\Big(\Lambda^r \nabla\times(B\times(\nabla\times B)),\Lambda^r B\Big)&\lesssim\|[\Lambda^{r},B\times](\nabla\times B)\|_{L^2}\|B\|_{\dot{H}^{r+1}}\\
&\lesssim\|\Lambda^{r}B\|_{L^{p_1}}\|\nabla B\|_{L^{p_2}}\|B\|_{\dot{H}^{r+1}}\\
&\lesssim\|B\|_{\dot{H}^{r+\frac{3}{p_2}}}\|B\|_{\dot{H}^{1+\frac{3}{p_1}}}\|B\|_{\dot{H}^{r+1}}\\
&\lesssim \Big(\|B\|_{{H}^{r}}^{1-\frac{3}{p_2}}\|\nabla B\|_{{H}^{r}}^{\frac{3}{p_2}}\Big)\Big(\|B\|_{{H}^{r}}^{r-\frac{3}{p_1}}\|\nabla B\|_{{H}^{r}}^{\frac{3}{p_1}+1-r}\Big)\|\nabla B\|_{{H}^{r}}\\
&\lesssim \|B\|_{{H}^{r}}^{r-\frac{1}{2}}\|\nabla B\|_{{H}^{r}}^{\frac{7}{2}-r},
\end{split}
\end{equation}where $\Lambda:=(-\Delta)^{\frac{1}{2}}$ and $\frac{1}{p_1}+\frac{1}{p_2}=\frac{1}{2}$ with $p_2\geq 3$. From above, the requirement $r>\frac{3}{2}$ seems inevitable due to the use of Young's inequality for splitting the last term. Thus, such an analysis inspires us to choose another method to deal with the case $B_0\in H^{\frac{3}{2}}(\mathbb{R}^3)$. Instead, we will use the classical continuity method to establish the uniform bounds in the proof of \thmref{thm1}, where the Fourier localization techniques and paradifferential calculus play a crucial role.
\end{rem}
Notice that the product estimate\begin{equation}\label{bades}
\|\nabla\times(u\times B)\|_{\dot{H}^{\frac{1}{2}}}\lesssim \|u\|_{\dot{H}^{\frac{3}{2}-\eta}}\|B\|_{\dot{H}^{\frac{3}{2}+\eta}}+\|u\|_{\dot{H}^{\frac{3}{2}+\theta}}\|B\|_{\dot{H}^{\frac{3}{2}-\theta}}
\end{equation}only holds for $\eta,\theta>0$, so in order to bound $\|\nabla\times(u\times B)\|_{L_t^2(\dot{H}^{\frac{1}{2}})}$ in the proof of \thmref{thm1}, the velocity $u$ should belong to $L^p_t(\dot{H}^{\frac{3}{2}+\theta})$ for some $p>2$. To attain this, we need to let initial velocity $u_0\in \dot{H}^{\frac{1}{2}+\sigma}$ for some $\sigma>\theta>0$ due to the parabolic smooth effect, which is the reason why the restriction $\sigma>0$ exists in \thmref{thm1}. Recently, Danchin and Tan in \cite{MR4193644} observed that if $\mu=\nu$, regarding $v:=u-\kappa\nabla\times B$ as a new unknown, the system \eqref{main eq} can be rewritten as (one can see \cite[P36-37]{MR4193644} for derivation):\begin{equation}\label{extend eq}
\begin{cases}
	\partial_tu-\mu\Delta u=B\cdot\nabla B-u\cdot\nabla u-\nabla P,\\
	\partial_tB-\mu\Delta B=\nabla\times (v\times B),\\
	\partial_tv-\mu\Delta v=B\cdot\nabla B-u\cdot\nabla u-\kappa\nabla\times((\nabla\times v)\times B)\\+\nabla\times(v\times u)+2\kappa\nabla\times (v\cdot\nabla B)-\nabla P,\\
	\Div u=\Div B=\Div v=0.
\end{cases}
\end{equation}Based on this new formulation, they prove the global well-posedness of \eqref{main eq} for small data in critical Besov spaces $\dot{B}^{\frac{1}{2}}_{2,r}(\mathbb{R}^3)\times (\dot{B}^{\frac{1}{2}}_{2,r}(\mathbb{R}^3)\cap\dot{B}^{\frac{3}{2}}_{2,r}(\mathbb{R}^3))$ with $1\leq r\leq\infty$ and local well-posedness for large data in $\dot{B}^{\frac{1}{2}}_{2,1}(\mathbb{R}^3)\times (\dot{B}^{\frac{1}{2}}_{2,1}(\mathbb{R}^3)\cap\dot{B}^{\frac{3}{2}}_{2,1}(\mathbb{R}^3))$. Actually, in the case of $\mu=\nu$, considering \eqref{extend eq} instead of \eqref{main eq} brings us more advantages. Specifically, as $\|B\|_{\dot{H}^{\frac{3}{2}}}\lesssim \|u\|_{\dot{H}^{\frac{1}{2}}}+\|v\|_{\dot{H}^{\frac{1}{2}}}$, it suffices to derive the $\dot{H}^{\frac{1}{2}}$-bound of $(u,B,v)$, which can help us avoid using the estimate \eqref{bades}. Besides, the quasilinear term $\kappa\nabla\times((\nabla\times v)\times B)$ has a cancelation property. By ultilizing \eqref{extend eq}, we indeed can reduce $\sigma$ to $0$ in \thmref{thm1} when $\mu=\nu$. Precisely, we have
\begin{thm}\label{thm2}
	Assume that $\mu=\nu$. Let $(u_0,B_0)\in H^{\frac{1}{2}}(\mathbb{R}^3)\times H^{\frac{3}{2}}(\mathbb{R}^3)$ with $\Div u_0=\Div B_0=0$. Then there exists a positive time $T$ such that the system \eqref{main eq} admits a unique solution $(u,B)$ with
	\begin{equation}
		\begin{split}
			u&\in C([0,T];H^{\frac{1}{2}}(\mathbb{R}^3))\cap L^2(0,T;H^{\frac{3}{2}}(\mathbb{R}^3)),\\
			B&\in C([0,T];H^{\frac{3}{2}}(\mathbb{R}^3))\cap L^2(0,T;H^{\frac{5}{2}}(\mathbb{R}^3)).
		\end{split}
	\end{equation}
\end{thm}
\begin{rem}\upshape
Because $\dot{B}^s_{2,1}(\mathbb{R}^3)$ and $H^s(\mathbb{R}^3)$ are not embedded into each other when $s>0$, \thmref{thm2} is different from the local well-posedness result in $\dot{B}^{\frac{1}{2}}_{2,1}(\mathbb{R}^3)\times (\dot{B}^{\frac{1}{2}}_{2,1}(\mathbb{R}^3)\cap\dot{B}^{\frac{3}{2}}_{2,1}(\mathbb{R}^3))$ mentioned above (see \cite[Theorem 2.2]{MR4193644}). 
\end{rem}
To the best of our knowledge, there are no available global well-posedness result concerning \eqref{main eq} with $B_0\in H^{\frac{3}{2}}(\mathbb{R}^3)$ for the general $\mu$ and $\nu$, so the natural next step is to extend the local strong solutions constructed in \thmref{thm1} globally in time. Actually, we find that this can be achieved for small initial data. Furthermore, we obtain the optimal time-decay rates of solutions.
\begin{thm}\label{thm3}
	Assume that the initial data $(u_0,B_0)$ satisfies the conditions in \thmref{thm1}. There exists a small positive constant $c_0=c_0(\mu,\nu,\kappa,\sigma)$ such that if 
	\begin{equation}\label{smallcon}
\|u_0\|_{H^{\frac{1}{2}+\sigma}}+\|B_0\|_{H^{\frac{3}{2}}}\leq c_0,
	\end{equation}then the system \eqref{main eq} has a unique global solution $(u,B)$ with  
\begin{equation}
	\begin{split}
		u&\in C([0,\infty);H^{\frac{1}{2}+\sigma}(\mathbb{R}^3)),\quad\nabla u\in L^2(0,\infty;H^{\frac{1}{2}+\sigma}(\mathbb{R}^3)),\\
		B&\in C([0,\infty);H^{\frac{3}{2}}(\mathbb{R}^3)),\quad\nabla B\in L^2(0,\infty;H^{\frac{3}{2}}(\mathbb{R}^3)).
	\end{split}
\end{equation}Moreover, if in addition $(u_0,B_0)\in\dot{B}^{-\gamma}_{2,\infty}(\mathbb{R}^3)$ for some $ \gamma\in[0,\frac{5}{2}]$, then it holds that for all $t\geq 1$,
\begin{equation}\label{decay}
	\|\Lambda^s u(t)\|_{L^2}+\|\Lambda^s B(t)\|_{L^2}\lesssim (1+t)^{-\frac{s+\gamma}{2}}\quad\text{if}\quad s\geq 0.
\end{equation}
\end{thm}
\begin{rem}\upshape
The time-decay rates \eqref{decay} are optimal in the sense that they concide with that of the heat semi-group.
\end{rem}
\begin{rem}\upshape
From the proof of \thmref{thm1}, we see that the solution $(u,B)$ belongs to all Sobolev spaces after $t>0$, so the time-decay rates \eqref{decay} make sense.
\end{rem}
\begin{rem}\upshape
Actually, we do not need to impose additional regularity on initial data when $\gamma=0$ because $L^2(\mathbb{R}^3)\hookrightarrow\dot{B}^{0}_{2,\infty}(\mathbb{R}^3)$.
\end{rem}
The rest of the paper is organized as follows. In section \ref{sec2}, we recall some basic facts on the Littlewood-Paley theory. Section \ref{sec3}, Section \ref{sec4} and Section \ref{section3} are devoted to the proof of \thmref{thm1}, \thmref{thm2} and \thmref{thm3}, respectively.\\[10pt] 
\textbf{Notation: }Throughout the paper, $C$ stands for a generic positive constant, which may vary on different lines. For brevity, we sometimes use $a\lesssim b$ to replace $a\leq Cb$. The notation $a\approx b$ means that both $a\lesssim b$ and $b\lesssim a$ hold. For any two operators $A$ and $B$, we denote by $[A,B]=AB-BA$ the commutator between $A$ and $B$. For any interval $I$ of $\mathbb{R}$, we denote by $C(I;X)$ the set of continuous functions on $I$ with values in $X$. For $p\in[1,+\infty]$, $L^p(I;X)$ denotes the set of measurable functions $u:I\rightarrow X$ such that $t\rightarrow \|u(t)\|_X$ belongs to $L^p(I)$. Also, we use the shorthand notation $L_T^p(X)$ for $L^p(0,T;X)$.

\section{Preliminaries}\label{sec2}
\quad In this section, we briefly introduce the Littlewood-Paley decomposition, the definition of the homogeneous Besov space and some related analysis tools. For more details, we refer readers to \cite{MR2768550}.\par 
Let $\varphi,\chi\in\mathcal{S}(\mathbb{R}^3)$ be two radial functions valued in the interval $[0,1]$ and supported in  $\mathcal{C}=\{\xi\in\mathbb{R}^3:\frac{3}{4}\leq|\xi|\leq\frac{8}{3}\}$ and  $\mathcal{B}=\{\xi\in\mathbb{R}^3:|\xi|\leq\frac{4}{3}\}$ respectively such that\begin{equation*} 
	\sum_{j\in\mathbb{Z}}\varphi(2^{-j}\xi)=1\,\,\,\,\text{in \,\,}\mathbb{R}^3 \backslash \{0\},
\end{equation*}and\begin{equation*} 
	\chi(\xi)+\sum_{j\geq 0}\varphi(2^{-j}\xi)=1\,\,\,\,\text{in \,\,}\mathbb{R}^3.
\end{equation*}The dyadic block $\dot{\Delta}_j$ and the low frequency cut-off operator  $\dot{S}_j$ with $j\in\mathbb{Z}$ are defined as follows\begin{equation*}
\begin{split}
\dot{\Delta}_j u&:=\varphi(2^{-j}D)u=2^{3j}\int_{\mathbb{R}^3}g(2^jy)u(x-y)dy\quad\text{with}\quad g=\mathcal{F}^{-1}\phi,\\
\dot{S}_j u&:=\chi(2^{-j}D)u=2^{3j}\int_{\mathbb{R}^3}\tilde{g}(2^jy)u(x-y)dy\quad\text{with}\quad \tilde{g}=\mathcal{F}^{-1}\chi.
\end{split}
\end{equation*}We denote by $\mathcal{S}_h^\prime(\mathbb{R}^3)$ the set of tempered distribution $u$ such that\begin{equation*}
	\lim_{j\rightarrow-\infty}\|\dot{S}_ju\|_{L^\infty}=0,
\end{equation*}then the following decompositions hold for all $u\in\mathcal{S}_h^\prime(\mathbb{R}^3)$:\begin{equation*}\begin{split}
	u=\sum_{j\in\mathbb{Z}}\dot{\Delta}_j u\quad\text{and}\quad
		\dot{S}_j u=\sum_{j^\prime\leq j-1}\dot{\Delta}_{j^\prime} u.
\end{split}
\end{equation*}Moreover, it is easy to check that the Littlewood-Paley decomposition satisfies the almost orthogonality:\begin{equation}\label{orthogon}
	\dot{\Delta}_k\dot{\Delta}_j u=0\quad\text{if}\quad|k-j|\geq2,\quad\dot{\Delta}_k(\dot{S}_{j-1}u\dot{\Delta}_jv)=0\quad\text{if}\quad|k-j|\geq5.
\end{equation}\par
The following well-known Berstein inequalities will be frequently used in this sequel.
\begin{lem}[{\cite[Lemma 2.1]{MR2768550}}]\label{berstein}
	Let $0<r<R$. There exists a constant $C$ such that for any nonnegative integer $k$, any couple $(p,q)\in[1,\infty]^2$ with $q\geq p\geq1$, and any function $u$ in $L^p(\mathbb{R}^3)$, there holds\begin{equation*}
		\begin{split}
			&\operatorname{supp}\hat{u}\subset\{\xi\in\mathbb{R}^3:|\xi|\leq\lambda R \}\Longrightarrow\sup_{|\alpha|=k}\|\partial^\alpha u\|_{L^q}\leq C^{k+1}\lambda^{k+3(\frac{1}{p}-\frac{1}{q})}\|u\|_{L^p},\\ 
			&\operatorname{supp}\hat{u}\subset\{\xi\in\mathbb{R}^3:\lambda r\leq|\xi|\leq\lambda R \}\Longrightarrow C^{-k-1}\lambda^{k}\|u\|_{L^p}\leq\sup_{|\alpha|=k}\|\partial^\alpha u\|_{L^p}\leq C^{k+1}\lambda^{k}\|u\|_{L^p}. 
		\end{split}
	\end{equation*}
\end{lem}
Let us recall the definition of the homogeneous Besov space.\begin{definition}[{\cite{MR2768550}}]
	Let $s\in\mathbb{R}$ and $1\leq p,r\leq\infty$, the homogeneous Besov space $\dot{B}^{s}_{p,r}(\mathbb{R}^3)$ is defined by\begin{equation*}
		\dot{B}^{s}_{p,r}(\mathbb{R}^3):=\{u\in\mathcal{S}^\prime_h(\mathbb{R}^3):\|u\|_{\dot{B}^s_{p,r}}<\infty\}
	\end{equation*}with\begin{equation*}
		\|u\|_{\dot{B}^s_{p,r}}:=\Big\|\Big(2^{js}\|\dot{\Delta}_j u\|_{L^p}\Big)_{j\in\mathbb{Z} }\Big\|_{\ell^r(\mathbb{Z})}.
	\end{equation*}
\end{definition}
\begin{rem}\upshape
	It is easy to verify that the norm $\|\cdot\|_{\dot{B}^s_{2,2}}$ and the classical homogeneous Sobolev norm $\|\cdot\|_{\dot{H}^s}:=\|\Lambda^s\cdot\|_{L^2}$ with $\Lambda^s:=(-\Delta)^{\frac{s}{2}}$ are equivalent, so $\dot{B}^s_{2,2}(\mathbb{R}^3)$ coincides with $\dot{H}^s(\mathbb{R}^3)$.
\end{rem} 
\begin{lem}[{\cite{MR2768550}}]\label{prop} Let $s\in\mathbb{R}$ and $1\leq p,r\leq \infty$, the following properties hold true:
	\begin{enumerate} 
		\item[(i)] Derivatives: for any $k$-th order derivative, it holds that\begin{equation*}
			\sup_{|\alpha|=k}\|\partial^\alpha u\|_{\dot{{B}}^{s}_{p,r}}\approx \|u\|_{\dot{{B}}^{s+k}_{p,r}}.
		\end{equation*}
	\item[(ii)] Embedding: for $1\leq p\leq \tilde{p}\leq \infty$ and $1\leq r\leq \tilde{r}\leq \infty$, one has\begin{equation*}\begin{split}
	&\dot{{B}}^{s}_{p,r}(\mathbb{R}^3)\hookrightarrow \dot{{B}}^{s-3(\frac{1}{p}-\frac{1}{\tilde{p}})}_{\tilde{p},\tilde{r}}(\mathbb{R}^3),\quad \dot{{B}}^{\frac{3}{p}}_{p,1}(\mathbb{R}^3)\hookrightarrow L^\infty(\mathbb{R}^3),\\
	&\quad\text{and}\quad L^p(\mathbb{R}^3)\hookrightarrow\dot{{B}}^{0}_{p,\infty}(\mathbb{R}^3). 
		\end{split}
	\end{equation*}
		\item[(iii)] Interpolation: for any  $s_1<s_2$ and $0<\theta<1$, there holds\begin{align}
				&\|u\|_{\dot{B}^{\theta s_1+(1-\theta)s_2}_{p,r}}\leq \|u\|^\theta_{\dot{B}^{s_1}_{p,r}}\|u\|^{1-\theta}_{\dot{B}^{s_2}_{p,r}},\label{inter1}\\
				&\|u\|_{\dot{B}^{\theta s_1+(1-\theta)s_2}_{p,1}}\leq\frac{C}{s_2-s_1}\Big(\frac{1}{\theta}+\frac{1}{1-\theta}\Big) \|u\|^\theta_{\dot{B}^{s_1}_{p,\infty}}\|u\|^{1-\theta}_{\dot{B}^{s_2}_{p,\infty}}.\label{inter2}
		\end{align}
		\item [(iv)] Action of multiplier: Let $f$ be a smooth function on $\mathbb{R}^n\backslash\{0\}$ which is homogeneous of degree $m$, then we have\begin{equation*}
			\|f(D)u\|_{\dot{B}^{s-m}_{p,r}}\lesssim\|u\|_{\dot{B}^{s}_{p,r}}.
		\end{equation*}
	\end{enumerate}		
\end{lem}\par
We need the following mixed time-space norm, which was introduced by Chemin in \cite{MR1753481}.
\begin{definition}
Let $T>0$, $s\in\mathbb{R}$ and $1\leq p,r,q\leq\infty$. For any tempered distribution $u$ on $(0,T)\times \mathbb{R}^3$, we set\begin{equation*}
		\|u\|_{\widetilde{L}^q_T(\dot{B}^s_{p,r})}:=\Big\|\Big(2^{js}\|\dot{\Delta}_j u\|_{L_T^q(L^p)}\Big)_{j\in\mathbb{Z}}\Big\|_{\ell^r(\mathbb{Z})}.
	\end{equation*}
\end{definition}
\begin{rem}\upshape\label{minkow}
By the Minkowski's inequality, one easily observe that\begin{equation*}
		\begin{cases}
			\|u\|_{\widetilde{L}^q_T(\dot{B}^s_{p,r})}\leq\|u\|_{{L}^q_T(\dot{B}^s_{p,r})}\quad \text{if}\quad q\leq r,  \\
			\|u\|_{{L}^q_T(\dot{B}^s_{p,r})}\leq\|u\|_{\widetilde{L}^q_T(\dot{B}^s_{p,r})}\quad \text{if}\quad q\geq r.
		\end{cases}
	\end{equation*}
\end{rem}
\par 
To obtain the nonlinear estimates in Besov spaces, we shall make use of the following Bony's decomposition from \cite{MR631751} in the homogeneous context:\begin{equation}\label{BD}
	uv=\dot{T}_uv+\dot{T}_vu+\dot{R}(u,v),
\end{equation}where \begin{equation*}
	\dot{T}_uv:=\sum_{j\in\mathbb{Z}}\dot{S}_{j-1}u\dot{\Delta}_jv\quad\text{and}\quad\dot{R}(u,v):=\sum_{j\in\mathbb{Z}}\sum_{|j^\prime-j|\leq 1}\dot{\Delta}_ju{\dot{\Delta}}_{j^\prime}v.
\end{equation*}The operators $\dot{T}$ and $\dot{R}$ above are called paraproduct operator and remainder operator, respectively.
\begin{lem}[\cite{MR2768550}]\label{product1}
	Let $s\in\mathbb{R}$, $t<0$ and $1\leq p,r\leq\infty$, there exists a constant $C$ such that
	\begin{equation}
		\begin{split}
			\|\dot{T}_uv\|_{\dot{B}^{s}_{p,r}}&\leq C \|u\|_{{L^{\infty}}}\|v\|_{\dot{B}^{s}_{p,r}}\quad\text{and}\quad\|\dot{T}_uv\|_{\dot{B}^{s+t}_{p,r}}\leq C\|u\|_{\dot{B}^{t}_{\infty,\infty}}\|v\|_{\dot{B}^{s}_{p,r}}.
		\end{split}
	\end{equation}
Let $s_1,s_2\in\mathbb{R}$ and $1\leq p,p_1,p_2,r,r_1,r_2\leq\infty$ with $\frac{1}{p}=\frac{1}{p_1}+\frac{1}{p_2}$ and $\frac{1}{r}= \frac{1}{r_1}+\frac{1}{r_2}$. If $s_1+s_2>0$, there exists a constant $C$ such that\begin{equation}
	\begin{split}
		\|\dot{R}(u,v)\|_{\dot{B}^{s_1+s_2}_{p,r}}\leq C\|u\|_{\dot{B}^{s_1}_{p_1,r_1}}\|v\|_{\dot{B}^{s_2}_{p_2,r_2}}.
	\end{split}
\end{equation}If $r=1$ and $s_1+s_2=0$, there exists a constant $C$ such that\begin{equation}\label{zero}
	\begin{split}
		\|\dot{R}(u,v)\|_{\dot{B}^{0}_{p,\infty}}\leq C\|u\|_{\dot{B}^{s_1}_{p_1,r_1}}\|v\|_{\dot{B}^{s_2}_{p_2,r_2}}.
	\end{split}
\end{equation}
\end{lem}
We have the following product law in Besov spaces.
\begin{lem}\label{product3} Let $s>-\frac{3}{2}$, $1\leq r\leq\infty$ and $\eta,\theta>0$, it holds that
\begin{equation}\label{pr2}
	\begin{split}
\|fg\|_{\dot{B}^s_{2,r}}&\lesssim \|f\|_{\dot{B}^{\frac{3}{2}-\eta}_{2,\infty}}\|g\|_{\dot{B}^{s+\eta}_{2,r}}+\|g\|_{\dot{B}^{\frac{3}{2}-\theta}_{2,\infty}}\|f\|_{\dot{B}^{s+\theta}_{2,r}}.
	\end{split}
\end{equation}
\end{lem}
\begin{proof}
It follows from the Bony's decompostion that\begin{equation}
	fg=\dot{T}_fg+\dot{T}_gf+\dot{R}(f,g).
\end{equation}Since $s>-\frac{3}{2}$ and $\eta,\theta>0$, by \lemref{product1} and embeddings (see \lemref{prop} (ii)), we have\begin{equation}
\begin{split}
\|\dot{T}_fg\|_{\dot{B}^s_{2,r}}&\lesssim \|f\|_{\dot{B}^{-\eta}_{\infty,\infty}}\|g\|_{\dot{B}^{s+\eta}_{2,r}}\lesssim \|f\|_{\dot{B}^{\frac{3}{2}-\eta}_{2,\infty}}\|g\|_{\dot{B}^{s+\eta}_{2,r}},\\
\|\dot{T}_gf\|_{\dot{B}^s_{2,r}}&\lesssim \|g\|_{\dot{B}^{-\theta}_{\infty,\infty}}\|f\|_{\dot{B}^{s+\theta}_{2,r}}\lesssim \|g\|_{\dot{B}^{\frac{3}{2}-\theta}_{2,\infty}}\|f\|_{\dot{B}^{s+\theta}_{2,r}},\\
\|\dot{R}(f,g)\|_{\dot{B}^s_{2,r}}&\lesssim \|\dot{R}(f,g)\|_{\dot{B}^{s+\frac{3}{2}}_{1,r}}\lesssim \|g\|_{\dot{B}^{\frac{3}{2}-\theta}_{2,\infty}}\|f\|_{\dot{B}^{s+\theta}_{2,r}}.
\end{split}
\end{equation}This completes the proof of \eqref{pr2}.
\end{proof}
We establish the following commutator estimates, which plays a key role in the analysis of the Hall term.
\begin{lem}\label{commutator}
Let $s>-\frac{3}{2}$, $1\leq r\leq\infty$, $0<\eta<\frac{5}{2}$ and $\theta>0$, the following estimate holds true,
\begin{equation}\label{commues}
\|2^{js}\|[\dot{\Delta}_j,f]g\|_{L^2}\|_{\ell^r(\mathbb{Z})}\lesssim \|f\|_{\dot{B}^{\frac{5}{2}-\eta}_{2,r}}\|g\|_{\dot{B}^{s+\eta-1}_{2,r}}+\|g\|_{\dot{B}^{\frac{3}{2}-\theta}_{2,\infty}}\|f\|_{\dot{B}^{s+\theta}_{2,r}}.
\end{equation}
\end{lem}
\begin{proof}
By the Bony's decomposition,\begin{equation}\label{decom1}
	[\dot{\Delta}_j,f]g=[\dot{\Delta}_j,\dot{T}_f]g+\dot{\Delta}_j(\dot{T}_gf+\dot{R}(f,g))-(\dot{T}_{\dot{\Delta}_jg}f+\dot{R}(\dot{\Delta}_jg,f)).
\end{equation}From \eqref{orthogon}, the first term can be written as\begin{equation}
	[\dot{\Delta}_j,\dot{T}_f]g=\sum_{|j-j^\prime|\leq 4} [\dot{\Delta}_j,\dot{S}_{j^\prime-1}f]\dot{\Delta}_{j^\prime}g.
\end{equation}According to \cite[Lemma 2.97]{MR2768550}, one has\begin{equation}
	\|[\dot{\Delta}_j,\dot{S}_{j^\prime-1}f]\dot{\Delta}_{j^\prime}g\|_{L^2}\lesssim 2^{-j}\|\nabla \dot{S}_{j^\prime-1}f\|_{L^\infty}\|\dot{\Delta}_{j^\prime}g\|_{L^2}.
\end{equation}Because $\eta>0$, we have\begin{equation}
	\begin{split}
		\|\nabla \dot{S}_{j^\prime-1}f\|_{L^\infty}&\leq \sum_{j^{\prime\prime}\leq j^\prime-2} \|\nabla \dot{\Delta}_{j^{\prime\prime}}f\|_{L^\infty}\\
		&\leq \|\nabla f\|_{\dot{B}^{-\eta}_{\infty,\infty}}\sum_{j^{\prime\prime}\leq j^\prime-2} 2^{j^{\prime\prime}\eta}\lesssim 2^{j^{\prime}\eta} \|\nabla f\|_{\dot{B}^{-\eta}_{\infty,\infty}}.
	\end{split}
\end{equation}Therefore,\begin{equation}\label{p1}\begin{split}
		2^{js}\|[\dot{\Delta}_j,\dot{T}_f]g\|_{L^2}
		&\leq2^{js}\sum_{|j-j^\prime|\leq 4}\|[\dot{\Delta}_j,\dot{S}_{j^\prime-1}f]\dot{\Delta}_{j^\prime}g\|_{L^2}\\&\lesssim \|\nabla f\|_{\dot{B}^{-\eta}_{\infty,\infty}}\sum_{|j-j^\prime|\leq 4}2^{j^\prime(s+\eta-1)}\|\dot{\Delta}_{j^\prime}g\|_{L^2}.
	\end{split}
\end{equation}Taking the $\ell^r(\mathbb{Z})$ norm for both sides of \eqref{p1} and using \lemref{prop} (i), (ii) implies that\begin{equation}\label{k1}\begin{split}
	\|2^{js}\|[\dot{\Delta}_j,\dot{T}_f]g\|_{L^2}\|_{\ell^r(\mathbb{Z})}&\lesssim \|\nabla f\|_{\dot{B}^{-\eta}_{\infty,\infty}}\|g\|_{\dot{B}^{s+\eta-1}_{2,r}}\lesssim \|f\|_{\dot{B}^{\frac{5}{2}-\eta}_{2,\infty}}\|g\|_{\dot{B}^{s+\eta-1}_{2,r}}.
\end{split}
\end{equation}Since $\theta>0$ and $s>-\frac{3}{2}$, we get by applying \lemref{product1} and embeddings that\begin{align}
	\|2^{js}\|\dot{\Delta}_j\dot{T}_gf\|_{L^2}\|_{\ell^r(\mathbb{Z})}&\lesssim\|g\|_{\dot{B}^{-\theta}_{\infty,\infty}}\|f\|_{\dot{B}^{s+\theta}_{2,r}}\lesssim\|g\|_{\dot{B}^{\frac{3}{2}-\theta}_{2,\infty}}\|f\|_{\dot{B}^{s+\theta}_{2,r}},\label{k2}\\
	\|2^{js}\|\dot{\Delta}_j\dot{R}(f,g)\|_{L^2}\|_{\ell^r(\mathbb{Z})}&\lesssim\|\dot{R}(f,g)\|_{\dot{B}^{s+\frac{3}{2}}_{1,r}}\lesssim\|g\|_{\dot{B}^{\frac{3}{2}-\theta}_{2,\infty}}\|f\|_{\dot{B}^{s+\theta}_{2,r}}.\label{kk}
\end{align}Note that\begin{equation}
\dot{T}_{\dot{\Delta}_jg}f+\dot{R}(\dot{\Delta}_jg,f)=\sum_{j^\prime\geq j-2}\dot{S}_{j^\prime+2}\dot{\Delta}_jg\dot{\Delta}_{j^\prime}f.
\end{equation}Hence, by using Berstein inequalities, we get\begin{equation}\label{clo}
\begin{split}
	2^{js}\|\dot{T}_{\dot{\Delta}_jg}f+\dot{R}(\dot{\Delta}_jg,f)\|_{L^2}&\lesssim \sum_{j^\prime\geq j-2}2^{js}\|\dot{\Delta}_jg\|_{L^\infty}\|\dot{\Delta}_{j^\prime}f\|_{L^2}\\
	&\lesssim \sum_{j^\prime\geq j-2}2^{j(s+\frac{3}{2})}\|\dot{\Delta}_jg\|_{L^2}\|\dot{\Delta}_{j^\prime}f\|_{L^2}\\
	&\lesssim \|g\|_{\dot{B}^{s+\eta-1}_{2,\infty}}\sum_{j^\prime\geq j-2}2^{(j-j^\prime)(\frac{5}{2}-\eta)}2^{j^\prime(\frac{5}{2}-\eta)}\|\dot{\Delta}_{j^\prime}f\|_{L^2}.
\end{split}
\end{equation}Since $\eta<\frac{5}{2}$, taking the $\ell^r(\mathbb{Z})$ norm of both sides of \eqref{clo} and using the convolution inequality, we have\begin{equation}
\begin{split}
\|2^{js}\|\dot{T}_{\dot{\Delta}_jg}f+\dot{R}(\dot{\Delta}_jg,f)\|_{L^2}\|_{\ell^r(\mathbb{Z})}&\lesssim \|g\|_{\dot{B}^{s+\eta-1}_{2,\infty}}\|f\|_{\dot{B}^{\frac{5}{2}-\eta}_{2,r}}
\end{split}
\end{equation}Thus, the proof of \eqref{commues} is finished.
\end{proof}
We end this section with the following classical Kato-Ponce commutator estimates and Sobolev embedding inequality.
\begin{lem}[\cite{MR951744}]\label{commues3}
	Let $p,p_1,p_3\in (1,\infty)$ and $p_2,p_4\in[1,\infty]$ satisfying $\frac{1}{p}=\frac{1}{p_1}+\frac{1}{p_2}=\frac{1}{p_3}+\frac{1}{p_4}$. Then for any $s>0$, there holds\begin{equation}
		\|[\Lambda^s,f]g\|_{L^p}\lesssim \|\Lambda^sf\|_{L^{p_1}}\|g\|_{L^{p_2}}+\|\Lambda^{s-1}g\|_{L^{p_3}}\|\nabla f\|_{L^{p_4}}.
	\end{equation}
\end{lem}

\begin{lem}[{\cite[Theorem 1.38]{MR2768550}}]\label{GN}
	Let $0\leq s< \frac{3}{2}$, the following embedding relation holds\begin{equation}
		\dot{H}^s(\mathbb{R}^3)\hookrightarrow L^{\frac{6}{3-2s}}(\mathbb{R}^3).\label{Sobolevem}
	\end{equation}
\end{lem}

\section{Proof of \thmref{thm1}}\label{sec3}
In this section, we are going to prove \thmref{thm1}. For simplicity, we hope to avoid tracking the dependency of $\mu,\nu,\kappa$ for the constants in the proofs, so we set $\mu=\nu=\kappa=1$ in the rest of this paper. For presentation clarity, we divide the proof into the following several steps.\par 
\noindent\textbf{First step: A priori estimates}\par
Our goal in this step is to establish the following statement.
\begin{prop}\label{aprior}
Let $0<\sigma<2$ and $T>0$. If $(u,B)$ solves the system \eqref{main eq} on $[0,T]$, then there exist postive constants $C_1>1$, $C_2$ and $c$ such that the following a priori estimates hold true,
\begin{equation}\label{apriores}
	\begin{split}
&\|u\|_{L^\infty_T({H}^{\frac{1}{2}+\sigma})}+\|\nabla u\|_{L^2_T({H}^{\frac{1}{2}+\sigma})}+\|B\|_{L^\infty_T({H}^{\frac{3}{2}})}+\|\nabla B\|_{L^2_T({H}^{\frac{3}{2}})}\\&\leq C_1\Big( \|u_0\|_{{H}^{\frac{1}{2}+\sigma}}+\|B_0\|_{{H}^{\frac{3}{2}}}+T^{\frac{\sigma}{4\sigma+6}}\|u\|_{L^\infty_T({H}^{\frac{1}{2}+\sigma})}^{\frac{3\sigma+3}{2\sigma+3}}\|\nabla u\|_{L^2_T({H}^{\frac{1}{2}+\sigma})}^{\frac{\sigma+3}{2\sigma+3}}\\&\quad+T^{\frac{1}{2}-\frac{\sigma}{4}}\|B\|_{L^\infty_T({H}^{\frac{3}{2}})}^{2-\frac{\sigma}{2}}\|\nabla B\|_{L_T^{2}({H}^{\frac{3}{2}})}^{\frac{\sigma}{2}} \\&\quad+T^{\frac{\sigma}{4}}\|u\|_{L^\infty_T({H}^{\frac{1}{2}+\sigma})}\|B\|_{L^\infty_T({H}^{\frac{3}{2}})}^{\frac{\sigma}{2}}\|\nabla B\|_{L^2_T({H}^{\frac{3}{2}})}^{1-\frac{\sigma}{2}}\\&\quad+T^{\frac{\sigma}{4}}\|u\|_{L^\infty_T({H}^{\frac{1}{2}+\sigma})}^{\frac{\sigma}{2}}\|\nabla u\|_{L^2_T({H}^{\frac{1}{2}+\sigma})}^{1-\frac{\sigma}{2}}\|B\|_{L^\infty_T({H}^{\frac{3}{2}})}+\|B\|_{L^4_T(\dot{H}^{2})}^2\Big),
	\end{split}
\end{equation}and\begin{equation}\label{apriores2}
\begin{split}
\|B\|_{{L}^4_T(\dot{H}^{2})}&\leq  C_2\Big(\Big[\sum_{j\in\mathbb{Z}}({1-e^{-4c2^{2j}T}})^{\frac{1}{2}}(2^{j\frac{3}{2}}\|\dot{\Delta}_jB_0\|_{L^2})^2\Big]^{\frac{1}{2}}\\&\quad+ T^{\frac{\sigma}{4}}\|u\|_{L^\infty_T({H}^{\frac{1}{2}+\sigma})}\|B\|_{L^\infty_T({H}^{\frac{3}{2}})}^{\frac{\sigma}{2}}\|\nabla B\|_{L^2_T({H}^{\frac{3}{2}})}^{1-\frac{\sigma}{2}}\\&\quad+T^{\frac{\sigma}{4}}\|u\|_{L^\infty_T({H}^{\frac{1}{2}+\sigma})}^{\frac{\sigma}{2}}\|\nabla u\|_{L^2_T({H}^{\frac{1}{2}+\sigma})}^{1-\frac{\sigma}{2}}\|B\|_{L^\infty_T({H}^{\frac{3}{2}})}+\|B\|_{L^4_T(\dot{H}^{2})}^2\Big).
\end{split}
\end{equation}
\end{prop}
\begin{proof}
Taking the inner product of $\eqref{main eq}$ with $u$ and $B$, respectively, and adding up the resulting equations, one easily have\begin{equation}\label{basicener}
	\frac{1}{2}\frac{d}{dt}(\|u\|_{L^2}^2+\|B\|_{L^2}^2)+\|\nabla u\|_{L^2}^2+\|\nabla B\|_{L^2}^2=0,
\end{equation}where we have used the fact that\begin{equation*}
	\int_{\mathbb{R}^3}\nabla\times((\nabla\times B)\times B)\cdot Bdx=\int_{\mathbb{R}^3}((\nabla\times B)\times B)\cdot (\nabla\times B)dx=0.
\end{equation*}Hence, integrating \eqref{basicener} over $[0,T]$ both sides leads to\begin{equation}\label{nice0}
\begin{split}
	\|u\|_{L^\infty_T(L^2)}+\|B\|_{L^\infty_T(L^2)}+\|\nabla u\|_{L^2_T(L^2)}+\|\nabla B\|_{L^2_T(L^2)}\lesssim\|u_0\|_{L^2}+\|B_0\|_{L^2}.
\end{split}
\end{equation}\par 
Applying the operator $\Lambda^{\frac{1}{2}+\sigma}$ to the first equation of \eqref{main eq} and taking the inner product with $\Lambda^{\frac{1}{2}+\sigma}u$, then using H\"{o}lder's inequality and Young's inequality, we obtain\begin{equation}\label{basic2}
\begin{split}
&\frac{1}{2}\frac{d}{dt}\|\Lambda^{\frac{1}{2}+\sigma}u\|_{L^2}^2+\|\Lambda^{\frac{3}{2}+\sigma}u\|_{L^2}^2\\&=-\Big(\Lambda^{\frac{1}{2}+\sigma}(u\cdot\nabla u),\Lambda^{\frac{1}{2}+\sigma} u\Big)+\Big(\Lambda^{\frac{1}{2}+\sigma}(B\cdot\nabla B),\Lambda^{\frac{1}{2}+\sigma} u\Big)\\
&\leq C\|\Lambda^{\sigma-\frac{1}{2}}(u\cdot\nabla u)\|_{L^2}^2+C\|\Lambda^{\sigma-\frac{1}{2}}(B\cdot\nabla B)\|_{L^2}^2+\frac{1}{2}\|\Lambda^{\frac{3}{2}+\sigma}u\|_{L^2}^2,
\end{split}
\end{equation}whence, integrating with respect to time gives\begin{equation}\label{gg1}
\begin{split}
	\|u\|_{L^\infty_T(\dot{H}^{\frac{1}{2}+\sigma})}&+\|\nabla u\|_{L^2_T(\dot{H}^{\frac{1}{2}+\sigma})}\lesssim\|u_0\|_{\dot{H}^{\frac{1}{2}+\sigma}}+\|u\cdot\nabla u\|_{L^2_T(\dot{H}^{\sigma-\frac{1}{2}})}+\|B\cdot\nabla B\|_{L^2_T(\dot{H}^{\sigma-\frac{1}{2}})}.
\end{split}
\end{equation}\par To derive the high-order estimates of $B$, we use the Fourier localization techniques. Applying the operator $\dot{\Delta}_j$ to the second equation in \eqref{main eq} and taking the inner product with $\dot{\Delta}_jB$, one has\begin{equation}\label{Bes}
\begin{split}
	\frac{1}{2}\frac{d}{dt}\|\dot{\Delta}_jB\|_{L^2}^2+\|\nabla\dot{\Delta}_j B\|_{L^2}^2&\leq \|\dot{\Delta}_j\nabla\times(u\times B)\|_{L^2}\|\dot{\Delta}_jB\|_{L^2}\\&\quad+\Big(\dot{\Delta}_j(\nabla\times(B\times(\nabla\times B) )),\dot{\Delta}_jB\Big).
\end{split}
\end{equation}Thanks to Berstein inequalities and the cancelation property, there exist universal constants $c$ and $C$ such that\begin{equation}\label{haosan}
\|\nabla \dot{\Delta}_jB\|_{L^2}^2 \geq c 2^{2j}\|\dot{\Delta}_jB\|_{L^2}^2,
\end{equation}and\begin{equation}\label{convo}
\begin{split}
	&\Big(\dot{\Delta}_j(\nabla\times(B\times(\nabla\times B) )),\dot{\Delta}_jB\Big)
	\\=&	\Big(\dot{\Delta}_j(B\times(\nabla\times B) )-B\times(\nabla\times\dot{\Delta}_j B) ,\nabla\times\dot{\Delta}_jB\Big)
	\\=&	\Big([\dot{\Delta}_j,B\times](\nabla\times B) ,\nabla\times\dot{\Delta}_jB\Big)\\
	\leq &\|[\dot{\Delta}_j,B\times](\nabla\times B)\|_{L^2}\|\nabla\times\dot{\Delta}_jB\|_{L^2}\\
	\leq &C2^j\|[\dot{\Delta}_j,B\times](\nabla\times B)\|_{L^2}\|\dot{\Delta}_jB\|_{L^2}
\end{split}
\end{equation}Inserting \eqref{haosan} and \eqref{convo} into \eqref{Bes}, and dividing through by $\|\dot{\Delta}_jB\|_{L^2}$, we find that\begin{equation}\label{1}
\begin{split}
	\frac{d}{dt}\|\dot{\Delta}_jB\|_{L^2}+c2^{2j}\|\dot{\Delta}_j B\|_{L^2}\leq& \|\dot{\Delta}_j\nabla\times(u\times B)\|_{L^2}+C2^j\|[\dot{\Delta}_j,B\times](\nabla\times B)\|_{L^2}.
\end{split}
\end{equation}Multiplying both sides of \eqref{1} by $e^{c2^{2j}t}$ gives\begin{equation}\label{2}
\begin{split}
\frac{d}{dt}\Big(e^{c2^{2j}t}\|\dot{\Delta}_jB\|_{L^2}\Big)\leq& e^{c2^{2j}t}\Big(\|\dot{\Delta}_j\nabla\times(u\times B)\|_{L^2}+C2^j\|[\dot{\Delta}_j,B\times](\nabla\times B)\|_{L^2}\Big).
\end{split}
\end{equation}After time integration, we get that for any $t\in [0,T]$,\begin{equation}\label{juanji}
\begin{split}
		\|\dot{\Delta}_jB(t)\|_{L^2}\leq& e^{-c2^{2j}t}\|\dot{\Delta}_jB_0\|_{L^2}+\int_{0}^{t}e^{-c2^{2j}(t-\tau)}\Big(\|\dot{\Delta}_j\nabla\times(u\times B)\|_{L^2}\\&+C2^j\|[\dot{\Delta}_j,B\times](\nabla\times B)\|_{L^2}\Big)d\tau.
\end{split}
\end{equation}Applying convolution inequalities to \eqref{juanji} implies that for any $p,q\in [2,\infty]$ with $1+\frac{1}{p}=\frac{1}{2}+\frac{1}{q}$, \begin{equation}\label{3}
\begin{split}
	\|\dot{\Delta}_jB\|_{L^p_T(L^2)}\leq& \Big(\frac{1-e^{-pc2^{2j}T}}{pc2^{2j}}\Big)^{\frac{1}{p}}\|\dot{\Delta}_jB_0\|_{L^2}+\Big(\frac{1-e^{-qc2^{2j}T}}{qc2^{2j}}\Big)^{\frac{1}{q}}\Big(\|\dot{\Delta}_j\nabla\times(u\times B)\|_{L^2_T(L^2)}\\&+C2^j\|[\dot{\Delta}_j,B\times](\nabla\times B)\|_{L^2_T(L^2)}\Big).
\end{split}
\end{equation}
Multiplying both sides of \eqref{3} by $2^{j(\frac{3}{2}+\frac{2}{p})}$ and taking the $\ell^2(\mathbb{Z})$ norm, we arrive at\begin{equation}\label{important}
	\begin{split}
		\|B&\|_{\widetilde{L}_T^p(\dot{H}^{\frac{3}{2}+\frac{2}{p}})} \leq \frac{1}{(pc)^{\frac{1}{p}}} \Big[\sum_{j\in\mathbb{Z}}({1-e^{-pc2^{2j}T}})^{\frac{2}{p}}(2^{j\frac{3}{2}}\|\dot{\Delta}_jB_0\|_{L^2})^2\Big]^{\frac{1}{2}}\\&+C\Big(\|\nabla\times(u\times B)\|_{L^2_T(\dot{H}^{\frac{1}{2}})}+\Big\|\|2^{j\frac{3}{2}}\|[\dot{\Delta}_j,B\times](\nabla\times B)\|_{L^2}\|_{\ell^2(\mathbb{Z})}\Big\|_{L^2_T}\Big),
	\end{split}
\end{equation}
Thus, due to fact that $\|f\|_{H^a}\approx\|f\|_{L^2}+\|f\|_{\dot{H}^a}$ and \remref{minkow}, we see that
\begin{equation}\label{gg2}
	\begin{split}
		&\|u\|_{L_T^\infty({H}^{\frac{1}{2}+\sigma})}+\|\nabla u\|_{L_T^2({H}^{\frac{1}{2}+\sigma})}+\|B\|_{L_T^\infty({H}^{\frac{3}{2}})}+\|\nabla B\|_{L_T^2({H}^{\frac{3}{2}})}\\&\lesssim\|u\|_{L_T^\infty(L^2)}+\|\nabla u\|_{L_T^2({L^2})}+\|B\|_{L_T^\infty(L^2)}+\|\nabla B\|_{L_T^2({L^2})}\\
		&\qquad+\|u\|_{{L}_T^\infty(\dot{H}^{\frac{1}{2}+\sigma})}+\| u\|_{{L}_T^2(\dot{H}^{\frac{3}{2}+\sigma})}+\|B\|_{\widetilde{L}_T^\infty(\dot{H}^{\frac{3}{2}})}+\|B\|_{\widetilde{L}_T^2(\dot{H}^{\frac{5}{2}})}\\&\lesssim \|u_0\|_{H^{\frac{1}{2}+\sigma}}+\|B_0\|_{H^{\frac{3}{2}}}+\|u\cdot\nabla u\|_{L^2_T(\dot{H}^{\sigma-\frac{1}{2}})}+\|B\cdot\nabla B\|_{L^2_T(\dot{H}^{\sigma-\frac{1}{2}})}\\&\qquad+\|\nabla\times(u\times B)\|_{L^2_T(\dot{H}^{\frac{1}{2}})}+\Big\|\|2^{j\frac{3}{2}}\|[\dot{\Delta}_j,B\times](\nabla\times B)\|_{L^2}\|_{\ell^2(\mathbb{Z})}\Big\|_{L^2_T},
	\end{split}
\end{equation}and\begin{equation}\label{gg3}
	\begin{split}
\|B\|_{L_T^4(\dot{H}^2)}&\leq \|B\|_{\widetilde{L}_T^4(\dot{H}^2)}\lesssim \Big[\sum_{j\in\mathbb{Z}}({1-e^{-4c2^{2j}T}})^{\frac{1}{2}}(2^{j\frac{3}{2}}\|\dot{\Delta}_jB_0\|_{L^2})^2\Big]^{\frac{1}{2}}\\&+\|\nabla\times(u\times B)\|_{L^2_T(\dot{H}^{\frac{1}{2}})}+\Big\|\|2^{j\frac{3}{2}}\|[\dot{\Delta}_j,B\times](\nabla\times B)\|_{L^2}\|_{\ell^2(\mathbb{Z})}\Big\|_{L^2_T}.
	\end{split}
\end{equation}
\par 
Let us now turn to the nonlinear estimates for the right-hand sides of \eqref{gg2} and \eqref{gg3}. Recall that $0<\sigma<2$, it is clear that the following two product laws stem from \lemref{product3} with $(s,r,\eta,\theta)=(\frac{1}{2}+\sigma,2,1-\frac{\sigma}{2},1-\frac{\sigma}{2})$ and $(s,r,\eta,\theta)=(\frac{3}{2},2,1-\frac{\sigma}{2},\frac{\sigma}{2})$ respectively:\begin{align}
\|fg\|_{\dot{H}^{\frac{1}{2}+\sigma}}&\lesssim \|f\|_{\dot{H}^{\frac{1}{2}+\frac{\sigma}{2}}}\|g\|_{\dot{H}^{\frac{3}{2}+\frac{\sigma}{2}}}+\|f\|_{\dot{H}^{\frac{3}{2}+\frac{\sigma}{2}}}\|g\|_{\dot{H}^{\frac{1}{2}+\frac{\sigma}{2}}},\label{ppr1}\\
\|fg\|_{\dot{H}^{\frac{3}{2}}}&\lesssim\|f\|_{\dot{H}^{\frac{1}{2}+\frac{\sigma}{2}}}\|g\|_{\dot{H}^{\frac{5}{2}-\frac{\sigma}{2}}}+\|f\|_{\dot{H}^{\frac{3}{2}+\frac{\sigma}{2}}}\|g\|_{\dot{H}^{\frac{3}{2}-\frac{\sigma}{2}}}.\label{ppr2}
\end{align}Due to the condition $\Div u=\Div B=0$, one has\begin{equation*}
\begin{split}
	&u\cdot\nabla u=\Div(u\otimes u),\quad  B\cdot\nabla B=\Div(B\otimes B).
\end{split}\end{equation*}Using \eqref{ppr1}, H\"{o}lder's inequality, interpolation and the embedding $H^a\hookrightarrow H^b\hookrightarrow \dot{H}^b$ for $a\geq b\geq 0$, the right-hand side of \eqref{gg1} can be bounded as follows\begin{equation}\label{target1}
\begin{split}
\|u\cdot\nabla u\|_{L^2_T(\dot{H}^{\sigma-\frac{1}{2}})}&\lesssim \|u\otimes u\|_{L^2_T(\dot{H}^{\frac{1}{2}+\sigma})}\\
&\lesssim \|u\|_{L^\infty_T(\dot{H}^{\frac{1}{2}+\frac{\sigma}{2}})}\|u\|_{L^2_T(\dot{H}^{\frac{3}{2}+\frac{\sigma}{2}})}\\
&\lesssim \|u\|_{L^\infty_T(\dot{H}^{\frac{1}{2}+\frac{\sigma}{2}})}\Big(\|u\|_{L^2_T(L^2)}^{\frac{\sigma}{2\sigma+3}}\|\nabla u\|_{L^2_T(\dot{H}^{\frac{1}{2}+\sigma})}^{\frac{\sigma+3}{2\sigma+3}}\Big)\\
&\lesssim \|u\|_{L^\infty_T(\dot{H}^{\frac{1}{2}+\frac{\sigma}{2}})}\Big(T^{\frac{\sigma}{4\sigma+6}}\|u\|_{L^\infty_T(L^2)}^{\frac{\sigma}{2\sigma+3}}\|\nabla u\|_{L^2_T(\dot{H}^{\frac{1}{2}+\sigma})}^{\frac{\sigma+3}{2\sigma+3}}\Big)\\
&\lesssim T^{\frac{\sigma}{4\sigma+6}}\|u\|_{L^\infty_T({H}^{\frac{1}{2}+\sigma})}^{\frac{3\sigma+3}{2\sigma+3}}\|\nabla u\|_{L^2_T({H}^{\frac{1}{2}+\sigma})}^{\frac{\sigma+3}{2\sigma+3}},\end{split}
\end{equation}and\begin{equation}\label{target11}
	\begin{split}
\|B\cdot\nabla B\|_{L^2_T(\dot{H}^{\sigma-\frac{1}{2}})}&\lesssim\|B\otimes B\|_{L^2_T(\dot{H}^{\frac{1}{2}+\sigma})}\\
&\lesssim \|B\|_{L^\infty_T(\dot{H}^{\frac{1}{2}+\frac{\sigma}{2}})}\|B\|_{L_T^{2}(\dot{H}^{\frac{3}{2}+\frac{\sigma}{2}})}\\
&\lesssim \|B\|_{L^\infty_T(\dot{H}^{\frac{1}{2}+\frac{\sigma}{2}})}\|B\|_{L_T^{2}(\dot{H}^{\frac{3}{2}})}^{1-\frac{\sigma}{2}}\|\nabla B\|_{L_T^{2}(\dot{H}^{\frac{3}{2}})}^{\frac{\sigma}{2}}\\
&\lesssim T^{\frac{1}{2}-\frac{\sigma}{4}}\|B\|_{L^\infty_T({H}^{\frac{3}{2}})}^{2-\frac{\sigma}{2}}\|\nabla B\|_{L_T^{2}({H}^{\frac{3}{2}})}^{\frac{\sigma}{2}}.\end{split}
\end{equation}Likewise, resorting to \eqref{ppr2}, H\"{o}lder's inequality and interpolation gives
\begin{equation}\label{target111}
\begin{split}
		&\|\nabla\times(u\times B)\|_{L^2_T(\dot{H}^{\frac{1}{2}})}
		\lesssim \|u\times B\|_{L^2_T(\dot{H}^{\frac{3}{2}})}\\
		&\lesssim  \|u\|_{L^\infty_T(\dot{H}^{\frac{1}{2}+\frac{\sigma}{2}})}\|B\|_{L^2_T(\dot{H}^{\frac{5}{2}-\frac{\sigma}{2}})}+\|u\|_{L^2_T(\dot{H}^{\frac{3}{2}+\frac{\sigma}{2}})}\|B\|_{L^\infty_T(\dot{H}^{\frac{3}{2}-\frac{\sigma}{2}})}\\
		&\lesssim \|u\|_{L^\infty_T(\dot{H}^{\frac{1}{2}+\frac{\sigma}{2}})}\Big(\|B\|_{L^2_T(\dot{H}^{\frac{3}{2}})}^{\frac{\sigma}{2}}\|\nabla B\|_{L^2_T(\dot{H}^{\frac{3}{2}})}^{1-\frac{\sigma}{2}}\Big)\\&\qquad+\Big(\|u\|_{L^2_T(\dot{H}^{\frac{1}{2}+\sigma})}^{\frac{\sigma}{2}}\|\nabla u\|_{L^2_T(\dot{H}^{\frac{1}{2}+\sigma})}^{1-\frac{\sigma}{2}}\Big)\|B\|_{L^\infty_T(\dot{H}^{\frac{3}{2}-\frac{\sigma}{2}})}
		\\&\lesssim T^{\frac{\sigma}{4}}\|u\|_{L^\infty_T({H}^{\frac{1}{2}+\sigma})}\|B\|_{L^\infty_T({H}^{\frac{3}{2}})}^{\frac{\sigma}{2}}\|\nabla B\|_{L^2_T({H}^{\frac{3}{2}})}^{1-\frac{\sigma}{2}}\\&\qquad+T^{\frac{\sigma}{4}}\|u\|_{L^\infty_T({H}^{\frac{1}{2}+\sigma})}^{\frac{\sigma}{2}}\|\nabla u\|_{L^2_T({H}^{\frac{1}{2}+\sigma})}^{1-\frac{\sigma}{2}}\|B\|_{L^\infty_T({H}^{\frac{3}{2}})}.
	\end{split}
\end{equation}
For the commutator term, applying \lemref{commutator} with $(s,r,\eta,\theta)=(\frac{3}{2},2,\frac{1}{2},\frac{1}{2})$ yields that\begin{equation}\label{target2}
	\begin{split}
\Big\|\|2^{j\frac{3}{2}}\|[\dot{\Delta}_j,B\times](\nabla\times B)\|_{L^2}\|_{\ell^2(\mathbb{Z})}\Big\|_{L^2_T}\lesssim&  \|B\|_{L^4_T(\dot{H}^{2})}\|\nabla\times B\|_{L^4_T(\dot{H}^{1})}\lesssim \|B\|_{L^4_T(\dot{H}^{2})}^2.
	\end{split}
\end{equation}
Thus, plugging \eqref{target1}-\eqref{target2} into \eqref{gg2} and \eqref{gg3} yields \eqref{apriores} and \eqref{apriores2} respectively.
\end{proof}

\noindent\textbf{Second step: Construction of approximate solutions}\par 
We will construct a sequence of smooth approximate solutions $\{(u_n,B_n)\}$. For each positive integer $n$, the spectral cut-off operator $\mathcal{J}_n$ is defined as\begin{equation*}
	\mathcal{F}(\mathcal{J}_nf)(\xi)=1_{\{|\xi|\leq n\}}\mathcal{F}(f)(\xi).
\end{equation*}We smooth out the initial data $(u_0,B_0)$ by $
u_{0,n}=\mathcal{J}_nu_{0}$ and $B_{0,n}=\mathcal{J}_nB_{0}$. It is clear that $(u_{0,n},B_{0,n})\in H^\infty$ where $H^\infty:=\bigcap_{s>0}H^s$, and \begin{equation}\label{tends}
	\begin{split}
(u_{0,n},B_{0,n})\rightarrow (u_{0},B_{0})\quad\text{in}\quad H^{\frac{1}{2}+\sigma}\times H^{\frac{3}{2}}.
	\end{split}
\end{equation}\par
Consider the Hall-MHD system supplemented with smooth initial data $(u_{0,n},B_{0,n})$:
\begin{equation}\label{trun}
	\begin{cases}
		\partial_tu_n-\Delta u_n+u_n\cdot\nabla u_n+\nabla P_n=B_n\cdot\nabla B_n,\\
		\partial_tB_n-\Delta B_n+u_n\cdot\nabla B_n+\nabla\times((\nabla\times B_n)\times B_n)=B_n\cdot\nabla u_n,\\
		\Div u_n=\Div B_n=0,\\
		u_n(x,0)=u_{0,n}(x),\quad B_n(x,0)=B_{0,n}(x).
	\end{cases}
\end{equation}According to the local well-posedness result for smooth initial data (see \cite{MR3208454}), the above system admits a unique solution $(u_n,B_n)\in C([0,T_n);H^\infty)$ on some maximal time interval $[0,T_n)$. Furthermore, the blow-up criterion in \cite[Theorem 2]{MR3186849} indicates that if the maximal time $T_n<\infty$, there holds\begin{equation}\label{blowup}
\int_{0}^{T_n}(\|u_n\|_{BMO}^2+\|\nabla B_n\|_{BMO}^2)dt=\infty.
\end{equation}
\\
\noindent\textbf{Third step: Uniform bounds for the approximate solutions}\par 
Notice that the definition of cut-off operator ensures $\|\mathcal{J}_nu_0\|_{{H}^{\frac{1}{2}+\sigma}}\leq\|u_0\|_{{H}^{\frac{1}{2}+\sigma}}$, $\|\mathcal{J}_nB_0\|_{{H}^{\frac{3}{2}}}\leq\|B_0\|_{{H}^{\frac{3}{2}}}$ and $\|\dot{\Delta}_j\mathcal{J}_nB_0\|_{L^2}\leq\|\dot{\Delta}_jB_0\|_{L^2}$. Hence, denote\begin{equation*}
	\begin{split}
		X_n(t)&:=\|u_n\|_{L^\infty_t({H}^{\frac{1}{2}+\sigma})}+\|\nabla u_n\|_{L^2_t({H}^{\frac{1}{2}+\sigma})}+\|B_n\|_{L^\infty_t({H}^{\frac{3}{2}})}+\|\nabla B_n\|_{L^2_t({H}^{\frac{3}{2}})},\\
		Y_n(t)&:=\|B_n\|_{{L}^4_t(\dot{H}^{2})},
	\end{split}
\end{equation*}we can deduce from \propref{aprior} that for all $t\in [0,T_n)$,\begin{equation}\label{fd}
\begin{split}
	X_n(t)&\leq C_1\Big( \|u_0\|_{{H}^{\frac{1}{2}+\sigma}}+\|B_0\|_{{H}^{\frac{3}{2}}}+t^{\frac{\sigma}{4\sigma+6}}X_n^2(t)+t^{\frac{1}{2}-\frac{\sigma}{4}}X_n^2(t)+2t^{\frac{\sigma}{4}}X_n^2(t)+Y_n^2(t)\Big),\\
	Y_n(t)&\leq C_2\Big( \Big[\sum_{j\in\mathbb{Z}}({1-e^{-4c2^{2j}t}})^{\frac{1}{2}}(2^{j\frac{3}{2}}\|\dot{\Delta}_jB_0\|_{L^2})^2\Big]^{\frac{1}{2}}+2t^{\frac{\sigma}{4}}X_n^2(t)+Y_n^2(t)\Big).
\end{split}
\end{equation}Next, we define\begin{equation}\label{fd1}
T_{n}^\prime:=\sup\Big\{t\in [0,T_n): X_n(t)\leq 2C_1(\|u_0\|_{{H}^{\frac{1}{2}+\sigma}}+\|B_0\|_{{H}^{\frac{3}{2}}}),\,\,Y_n(t)\leq \delta  \Big\},
\end{equation}where the constant $\delta$ will be determined later, and choose a positive constant $T_*$ such that\begin{equation}\label{fd2}
\begin{cases}
T_*^{\frac{\sigma}{4\sigma+6}}4C_1^2(\|u_0\|_{{H}^{\frac{1}{2}+\sigma}}+\|B_0\|_{{H}^{\frac{3}{2}}})^2\leq\frac{1}{8}(\|u_0\|_{{H}^{\frac{1}{2}+\sigma}}+\|B_0\|_{{H}^{\frac{3}{2}}}),\\
T_*^{\frac{1}{2}-\frac{\sigma}{4}}4C_1^2(\|u_0\|_{{H}^{\frac{1}{2}+\sigma}}+\|B_0\|_{{H}^{\frac{3}{2}}})^2\leq\frac{1}{8}(\|u_0\|_{{H}^{\frac{1}{2}+\sigma}}+\|B_0\|_{{H}^{\frac{3}{2}}}),\\
\Big[\sum_{j\in\mathbb{Z}}({1-e^{-4c2^{2j}T_*}})^{\frac{1}{2}}(2^{j\frac{3}{2}}\|\dot{\Delta}_jB_0\|_{L^2})^2\Big]^{\frac{1}{2}}\leq\frac{\delta}{8C_2},\\
T_*^{\frac{\sigma}{4}}4C_1^2(\|u_0\|_{{H}^{\frac{1}{2}+\sigma}}+\|B_0\|_{{H}^{\frac{3}{2}}})^2\leq\frac{\delta}{8C_2}.
\end{cases}
\end{equation}Note that $T_*$ is a fixed constant independent of $n$. Let us explain why the third condition in \eqref{fd2} can be fulfilled. Actually, since $
\|B_0\|_{\dot{H}^{\frac{3}{2}}}=\Big[\sum_{j\in\mathbb{Z}}(2^{j\frac{3}{2}}\|\dot{\Delta}_jB_0\|_{L^2})^2\Big]^{\frac{1}{2}}<\infty
$, taking $N_0$ large enough can ensure that\begin{equation}\label{sum1}
\Big[\sum_{|j|>N_0}(2^{j\frac{3}{2}}\|\dot{\Delta}_jB_0\|_{L^2})^2\Big]^{\frac{1}{2}}\leq\frac{\delta}{16C_2}.
\end{equation}For the remaining finite terms, choosing $T_{*}$ sufficiently small such that\begin{equation}\label{sum2}
\Big[\sum_{|j|\leq N_0}({1-e^{-4c2^{2j}T_*}})^{\frac{1}{2}}(2^{j\frac{3}{2}}\|\dot{\Delta}_jB_0\|_{L^2})^2\Big]^{\frac{1}{2}}\leq\frac{\delta}{16C_2}.
\end{equation}Therefore,\begin{equation}\label{sum3}
\begin{split}
	&\Big[\sum_{j\in\mathbb{Z}}({1-e^{-4c2^{2j}T_*}})^{\frac{1}{2}}(2^{j\frac{3}{2}}\|\dot{\Delta}_jB_0\|_{L^2})^2\Big]^{\frac{1}{2}}\\&\leq \Big[\sum_{|j|\leq N_0}({1-e^{-4c2^{2j}T_*}})^{\frac{1}{2}}(2^{j\frac{3}{2}}\|\dot{\Delta}_jB_0\|_{L^2})^2\Big]^{\frac{1}{2}}\\&\quad+ \Big[\sum_{|j|>N_0}(2^{j\frac{3}{2}}\|\dot{\Delta}_jB_0\|_{L^2})^2\Big]^{\frac{1}{2}}\\&\leq \frac{\delta}{8C_2}.
\end{split}
\end{equation}\par We assert that when $\delta$ is sufficiently small, there holds $T_*<T_n^\prime$. Actually, if $T_n^\prime\leq T_*<\infty$, it follows from \eqref{fd}, \eqref{fd1} and \eqref{fd2} that for all $t\in[0,T_n^\prime)$,\begin{equation}\label{hahauni}
\begin{split}
	X_n(t)&\leq\frac{5C_1}{4}\Big( \|u_0\|_{{H}^{\frac{1}{2}+\sigma}}+\|B_0\|_{{H}^{\frac{3}{2}}}\Big)+\frac{C_1\delta }{4C_2}+C_1\delta^2,\\
	Y_n(t)&\leq \frac{3\delta}{8}+C_2\delta^2.
\end{split}
\end{equation}Taking $\delta$ small enough such that\begin{equation*}
\begin{split}
\frac{C_1\delta}{4C_2}+C_1\delta^2\leq \frac{C_1}{4}( \|u_0\|_{{H}^{\frac{1}{2}+\sigma}}+\|B_0\|_{{H}^{\frac{3}{2}}})\quad\text{and}\quad C_2\delta\leq \frac{1}{8},
\end{split}
\end{equation*}which leads to that for all $t\in[0,T_n^\prime)$,\begin{equation}
 	\begin{split}
 		X_n(t)&\leq\frac{3C_1}{2}\Big( \|u_0\|_{{H}^{\frac{1}{2}+\sigma}}+\|B_0\|_{{H}^{\frac{3}{2}}}\Big),\\
 		Y_n(t)&\leq \frac{\delta}{2}.
 	\end{split}
 \end{equation}At this stage, by the definition of $T_n^\prime$ and the continuity of $X_n(t)$ and $Y_n(t)$, one may conclude that $T_n=T_n^\prime<\infty$. However, due to the embedding $\dot{H}^{\frac{3}{2}}(\mathbb{R}^3)\hookrightarrow BMO(\mathbb{R}^3)$ (see \cite[Theorem 1.48]{MR2768550}),\begin{equation}\label{blowtime}
	\begin{split}
		\int_{0}^{T_n}(\|u_n\|_{BMO}^2+\|\nabla B_n\|_{BMO}^2)dt&\leq C\int_{0}^{T_n}(\|\nabla u_n\|_{\dot{H}^{\frac{1}{2}}}^2+\|\nabla B_n\|_{\dot{H}^{\frac{3}{2}}}^2)dt\\
	&\leq C\sup_{t\in [0,T_{n}^\prime)}X_n^2(t)<\infty,
	\end{split}
\end{equation}which contradicts with \eqref{blowup}. Thus, $T_*<T_n^\prime\leq T_n$ for all $n\geq 1$. By the definition of $T_n^\prime$, we have for all $n\geq 1$,\begin{equation}\label{realunibound}
\begin{split}
&\|u_n\|_{L^\infty_{T_*}({H}^{\frac{1}{2}+\sigma})}+\|\nabla u_n\|_{L^2_{T_*}({H}^{\frac{1}{2}+\sigma})}+\|B_n\|_{L^\infty_{T_*}({H}^{\frac{3}{2}})}+\|\nabla B_n\|_{L^2_{T_*}({H}^{\frac{3}{2}})}\\
&\leq 2C_1(\|u_0\|_{{H}^{\frac{1}{2}+\sigma}}+\|B_0\|_{{H}^{\frac{3}{2}}}).
\end{split}
\end{equation}

\noindent\textbf{Fourth step: Existence and uniqueness}\par 
We first show that the solution sequence $\{(u_n,B_n)\}$ constructed in the second step tends to some limit $(u,B)$ in $C([0,T_*];L^2)$. Making the difference between the system \eqref{trun} with $m$ and $n$, and taking the inner product with $u_m-u_n$ and $B_m-B_n$ respectively, one has\begin{equation}\label{total}
	\begin{split}
&\frac{1}{2}\frac{d}{dt}\Big(\|u_m-u_n\|_{L^2}^2+\|B_m-B_n\|_{L^2}^2\Big)+\|\nabla(u_m-u_n)\|_{L^2}^2+\|\nabla(B_m-B_n)\|_{L^2}^2\\
=&-\Big(u_m\cdot\nabla u_m-u_n\cdot\nabla u_n,u_m-u_n\Big)+\Big(B_m\cdot\nabla B_m-B_n\cdot\nabla B_n,u_m-u_n\Big)\\
&-\Big(u_m\cdot\nabla B_m-u_n\cdot\nabla B_n,B_m-B_n\Big)+\Big(B_m\cdot\nabla u_m-B_n\cdot\nabla u_n,B_m-B_n\Big)\\
&-\Big(\nabla\times((\nabla\times B_m)\times B_m))-\nabla\times((\nabla\times B_n)\times B_n)),B_m-B_n\Big)\\
=&A_1+A_2+A_3+A_4+A_5.
	\end{split}
\end{equation}We split $A_2$ and $A_4$ as follows:\begin{equation}
\begin{split}
A_2=&\Big((B_m-B_n)\cdot\nabla B_m,u_m-u_n\Big)+\Big(B_n\cdot\nabla (B_m-B_n),u_m-u_n\Big),\\
=&A_{2,1}+A_{2,2},\\
A_4=&\Big((B_m-B_n)\cdot\nabla u_m,B_m-B_n\Big)+\Big(B_n\cdot\nabla (u_m-u_n),B_m-B_n\Big)\\
=&A_{4,1}+A_{4,2}.
\end{split}
\end{equation}
Applying H\"{o}lder's inequality, \lemref{GN} and Young's inequality gives
\begin{equation}
	\begin{split}
		A_{2,1}&\leq \|B_m-B_n\|_{L^2}\|\nabla B_m\|_{{L}^3}\| u_m-u_n\|_{L^6}\\
		&\leq C\|B_n-B_m\|_{L^2}\|\nabla B_m\|_{{H}^{\frac{1}{2}}}\|\nabla (u_n-u_m)\|_{L^2}\\
		&\leq C\|\nabla B_m\|_{{H}^{\frac{1}{2}}}^2\|B_m-B_n\|_{L^2}^2+\frac{1}{8}\|\nabla (u_m-u_n)\|_{L^2}^2
	\end{split}
\end{equation}
Proceeding in a similar way,
\begin{equation}
	\begin{split}
		A_{4,1}&\leq C\|\nabla u_m\|_{{H}^{\frac{1}{2}}}^2\|B_n-B_m\|_{L^2}^2+\frac{1}{8}\|\nabla (B_m-B_n)\|_{L^2}^2.
	\end{split}
\end{equation}
By the divergence-free property, it is easily seen that $A_{2,2}+A_{4,2}=0$. As a result,
\begin{equation}\label{total1}
	\begin{split}
		A_2+A_4=&A_{2,1}+A_{4,1}\\
		\leq&C(\|\nabla B_m\|_{{H}^{\frac{3}{2}}}^2+\|\nabla u_m\|_{{H}^{\frac{1}{2}+\sigma}}^2)\|B_m-B_n\|_{L^2}^2
	\\	&+\frac{1}{8}\|\nabla (u_m-u_n)\|_{L^2}^2+\frac{1}{8}\|\nabla (B_m-B_n)\|_{L^2}^2
	\end{split}
\end{equation}Similarly,\begin{equation}\label{total2}
	\begin{split}
		A_1\leq&C\|\nabla u_m\|_{{H}^{\frac{1}{2}+\sigma}}^2\|u_m-u_n\|_{L^2}^2+\frac{1}{8}\|\nabla (u_m-u_n)\|_{L^2}^2,\\
		A_3\leq&C\|\nabla B_m\|_{{H}^{\frac{3}{2}}}^2\|u_m-u_n\|_{L^2}^2+\frac{1}{8}\|\nabla (B_m-B_n)\|_{L^2}^2.
	\end{split}
\end{equation}For the last term $A_5$, we rewrite it as\begin{equation}
\begin{split}
A_5=&-\Big(\nabla\times((\nabla\times (B_m-B_n))\times B_m),B_m-B_n\Big)\\
&-\Big(\nabla\times((\nabla\times B_n)\times (B_m-B_n)),B_m-B_n\Big)\\
=&A_{5,1}+A_{5,2}.
\end{split}
\end{equation}
Obviously,
\begin{equation}
	\begin{split}
		A_{5,1}=-\Big((\nabla\times (B_m-B_n))\times B_m,\nabla\times(B_m-B_n)\Big)=0.
	\end{split}
\end{equation}
Therefore, we get by using H\"{o}lder's inequality, \lemref{GN}, interpolation and Young's inequality that
\begin{equation}\label{total3}
	\begin{split}
		A_5&=A_{5,2}\\&=-\Big((\nabla\times B_n)\times (B_m-B_n),\nabla\times(B_m-B_n)\Big)\\
		&\leq \|\nabla\times B_n\|_{L^4}\|B_m-B_n\|_{L^4}\|\nabla\times(B_m-B_n)\|_{L^2}\\
		&\leq C\|B_n\|_{\dot{H}^{\frac{7}{4}}}\|B_m-B_n\|_{\dot{H}^{\frac{3}{4}}}\|\nabla(B_m-B_n)\|_{L^2}\\
		&\leq C\Big(\|B_n\|_{\dot{H}^{\frac{3}{2}}}^{\frac{3}{4}}\|\nabla B_n\|_{\dot{H}^{\frac{3}{2}}}^{\frac{1}{4}}\Big)\Big(\|B_m-B_n\|_{{L}^{2}}^{\frac{1}{4}}\|\nabla (B_m-B_n)\|_{L^2}^{\frac{3}{4}}\Big)\|\nabla(B_m-B_n)\|_{L^2}\\
		&\leq C\|B_n\|_{{H}^{\frac{3}{2}}}^{\frac{3}{4}}\|\nabla B_n\|_{{H}^{\frac{3}{2}}}^{\frac{1}{4}}\|B_m-B_n\|_{{L}^{2}}^{\frac{1}{4}}\|\nabla(B_m-B_n)\|_{L^2}^{\frac{7}{4}}\\
		&\leq C\|B_n\|_{H^{\frac{3}{2}}}^{6}\|\nabla B_n\|_{{H}^{\frac{3}{2}}}^{2}\|B_m-B_n\|_{{L}^{2}}^{2}+\frac{1}{8}\|\nabla(B_m-B_n)\|_{L^2}^{2}.
	\end{split}
\end{equation}Plugging \eqref{total1}, \eqref{total2} and \eqref{total3} into \eqref{total}, we arrive at\begin{equation}\label{totally}
\begin{split}
	&\frac{d}{dt}\Big(\|u_m-u_n\|_{L^2}^2+\|B_m-B_n\|_{L^2}^2\Big)+\|\nabla(u_m-u_n)\|_{L^2}^2+\|\nabla(B_m-B_n)\|_{L^2}^2\\
	&\leq C\Big(\|\nabla B_m\|_{{H}^{\frac{3}{2}}}^2+\|\nabla u_m\|_{{H}^{\frac{1}{2}+\sigma}}^2+\|B_n\|_{H^{\frac{3}{2}}}^{6}\|\nabla B_n\|_{{H}^{\frac{3}{2}}}^{2}\Big)\Big(\|u_m-B_n\|_{L^2}^2+\|B_m-B_n\|_{L^2}^2\Big).
\end{split}
\end{equation}By Gronwall's inequality and uniform bounds in \eqref{realunibound}, one has\begin{equation}\label{totally1}
\begin{split}
\sup_{0\leq t\leq T_*}\|u_m-u_n\|_{L^2}^2+\|B_m-B_n\|_{L^2}^2\leq C\Big(\|u_{0,m}-u_{0,n}\|_{L^2}^2+\|B_{0,m}-B_{0,n}\|_{L^2}^2\Big)
\end{split}
\end{equation}for some constant $C=C(\|u_0\|_{{H}^{\frac{1}{2}+\sigma}},\|B_0\|_{{H}^{\frac{3}{2}}})$. Obviously, the right-hand side above tends to zero as $n,m\rightarrow \infty$, which means that $\{(u_n,B_n)\}$ is a Cauchy sequence in $C([0,T_*];L^2)$. Consequently, there is a function pair $(u,B)\in C([0,T_*];L^2)$ such that\begin{equation}
(u_n,B_n)\rightarrow (u,B)\quad\text{in}\quad C([0,T_*];L^2),
\end{equation}From \eqref{totally}, we also get\begin{equation}
(\nabla u_n,\nabla B_n)\rightarrow (\nabla u,\nabla B)\quad\text{in}\quad L^2(0,T_*;L^2).
\end{equation}Using the uniform bounds in \eqref{realunibound} and interpolation, one can easily show that for any $s_1\in(0,\frac{1}{2}+\sigma)$ and any $s_2\in(0,\frac{3}{2})$,\begin{equation}\label{hello1}
\begin{split}
	u_n\rightarrow u\quad&\text{in}\quad C([0,T_*];H^{s_1}),\\
	\nabla u_n\rightarrow \nabla u\quad&\text{in}\quad L^2(0,T_*;H^{s_1}),\\
	B_n\rightarrow B\quad&\text{in}\quad C([0,T_*];H^{s_2}),\\
	\nabla B_n\rightarrow \nabla B\quad&\text{in}\quad L^2(0,T_*;H^{s_2}).
\end{split}
\end{equation}
Furthermore, due to the uniform bounds in \eqref{realunibound}, the Banach-Alaoglu theorem ensures that there exists a subsequence $\{(u_{n,k},B_{n,k})\}$ such that\begin{equation}\label{hello2}
	\begin{split}
		u_{n,k}\rightharpoonup u\quad&\text{weakly* in}\quad L^\infty(0,T_*;H^{\frac{1}{2}+\sigma})\cap L^2(0,T_*;H^{\frac{3}{2}+\sigma}),\\ B_{n,k}\rightharpoonup B\quad&\text{weakly* in}\quad L^\infty(0,T_*;H^{\frac{3}{2}})\cap L^2(0,T_*;H^{\frac{5}{2}}).
	\end{split}
\end{equation}\par
In order to verify the limit $(u,B)$ is indeed the solution to \eqref{main eq}, we need to pass to the limit in the nonlinear terms in \eqref{trun}. This procedure is relatively simple, as an example, we only deal with the Hall term, the other nonlinear terms can be treated by a similar way. Observe that\begin{equation}
\begin{split}
&\nabla\times((\nabla\times B_n)\times B_n)-\nabla\times((\nabla\times B)\times B)\\
=&\nabla\times((\nabla\times (B_n-B))\times B_n)+\nabla\times((\nabla\times B)\times (B_n-B)).
\end{split}
\end{equation}Using H\"{o}lder's inequality and \lemref{GN} gives\begin{equation}
\begin{split} 
\|\nabla\times((\nabla\times (B_n-B))\times B_n)\|_{\dot{H}^{-1}}&\leq C \|\nabla\times (B_n-B)\|_{L^3}\|B_n\|_{L^6}\\
&\leq C\|\nabla(B_n-B)\|_{{H}^{\frac{1}{2}}}\|B_n\|_{H^1},\\\|\nabla\times((\nabla\times B)\times (B_n-B))\|_{\dot{H}^{-1}}&\leq C\|\nabla\times B\|_{L^3}\|B_n-B\|_{L^6}\\
&\leq C\|\nabla B\|_{{H}^{\frac{1}{2}}}\|B_n-B\|_{H^1},
\end{split}
\end{equation}In view of \eqref{realunibound}, \eqref{hello1} and \eqref{hello2}, we have\begin{equation*}
\nabla\times((\nabla\times B_n)\times B_n)\rightarrow\nabla\times((\nabla\times B)\times B)\quad\text{in}\quad L^2(0,T_*;\dot{H}^{-1}).
\end{equation*}Thus, we end up with that the limit $(u,B)$ solves the Hall-MHD system \eqref{main eq} with some suitable pressure function $P$. The proof of uniqueness is quite similar to the $L^2$-convergence procedure above and so is omitted.\par 
What is left is to prove the continuity of $(u,B)$ in $H^{\frac{1}{2}+\sigma}\times H^{\frac{3}{2}}$ with respect to time. We see that $u$ satisfies the heat equation\begin{equation}
	\begin{split}
\partial_tu-\Delta u=\mathcal{P}(B\cdot\nabla B-u\cdot\nabla u),
	\end{split}
\end{equation}where $\mathcal{P}:=\Id+(-\Delta)^{-1}\nabla\Div$ is the Leray projection operator. Since $\mathcal{P}$ is bounded on Sobolev spaces, it follows from \eqref{target1} and \eqref{target11} that the right-hand side above belongs to $L^2(0,T_*;\dot{H}^{\sigma-\frac{1}{2}})$. Therefore, by the standard property of solutions to the linear heat equation (see \cite[Proposition A.4]{MR4193644}), we can deduce that $u\in C([0,T_*];H^{\frac{1}{2}+\sigma})$. However, the classical product laws do not allow us to prove that $\nabla\times((\nabla\times B)\times B)$ belongs $L^2(0,T_*;\dot{H}^{\frac{1}{2}})$, so we can not obtain the continuity of $B$ in $H^{\frac{3}{2}}$ by the way mentioned above. Observe that $(u(t),B(t))\in H^{\frac{3}{2}+\sigma}\times H^{\frac{5}{2}}$ almost everywhere on $[0,T_*]$ since $(u,B)\in L^2(0,T_*;H^{\frac{3}{2}+\sigma}\times H^{\frac{5}{2}})$. Hence, for any $\delta>0$, there exists some $T_1\in (0,\delta)$ such that $(u(T_1),B(T_1))\in H^{r}\times H^{r}$ with $r=\min\{\frac{3}{2}+\sigma,\frac{5}{2}\}>\frac{3}{2}$. Taking $(u(T_1),B(T_1))$ as initial data for \eqref{main eq} at $t=T_1$, according to the local well-posedness result for the Hall-MHD system in $H^s$ with $s>\frac{3}{2}$ (see \cite{MR4060362,MR3770153}), we obtain a solution $(\widetilde{u},\widetilde{B})$ on some time interval $[T_1,T_1^\prime]\subset [0,T_*]$ with $(\widetilde{u},\widetilde{B})\in C([T_1,T_1^\prime]; H^r)\cap L^2(T_1,T_1^\prime; H^{r+1})$. Thanks to the uniqueness stated above, we have $(u,B)=(\widetilde{u},\widetilde{B})$ on $[T_1,T_1^\prime]$. By a same way, we also can get that there exists some $T_2\in (T_1,\delta)$ such that $(u,B)\in C([T_2,T_2^\prime]; H^{r+1})\cap L^2(T_2,T_2^\prime; H^{r+2})$ on some time interval $[T_2,T_2^\prime]\subset [0,T_*]$. Arguing as in \eqref{blowtime} (the blow-up criterion \eqref{blowup} is corresponding to the local smooth solution in $H^s$ with $s>\frac{5}{2}$, see \cite[Theorem 2]{MR3186849}), $T_2^\prime$ can be choosen to be equal to $T_*$, which together with the arbitrariness of $\delta$ yields that $(u,B)\in C((0,T_*];H^{r+1})$. It is not hard to see that by iteration, we can prove that $(u,B)\in C((0,T_*];H^{\infty})$. On the other hand, remember that $B\in C([0,T_*];H^{s_2})$ with $s_2<\frac{3}{2}$ and $B\in L^\infty(0,T_*;H^{\frac{3}{2}})$, since $H^{-s_2}$ is dense in $H^{-\frac{3}{2}}$, by means of an $\frac{\epsilon}{2}$-argument, it is easy to show that $B$ is continuous in the weak topology of $H^{\frac{3}{2}}$ with respect to time, which implies that\begin{equation}
\varliminf _{t\rightarrow 0^+}\|B(t)\|_{H^{\frac{3}{2}}}\geq \|B_0\|_{H^{\frac{3}{2}}}.
\end{equation}Besides, it follows from \eqref{important} with $p=\infty$ that for all $t\in (0,T_*]$,\begin{equation}
 \begin{split}
 	\|B\|_{{L}_t^\infty(\dot{H}^{\frac{3}{2}})}\leq& \|B_0\|_{\dot{H}^{\frac{3}{2}}}+C\Big(\|B\cdot\nabla u\|_{L^2_t(\dot{H}^{\frac{1}{2}})}+\|u\cdot\nabla B\|_{L^2_t(\dot{H}^{\frac{1}{2}})}\\&+\Big\|\|2^{j\frac{3}{2}}\|[\dot{\Delta}_j,B\times](\nabla\times B)\|_{L^2}\|_{\ell^2(\mathbb{Z})}\Big\|_{L_t^2} \Big),
 \end{split}
\end{equation}so taking $t\rightarrow 0^+$ gives\begin{equation}
 	\varlimsup _{t\rightarrow 0^+}\|B(t)\|_{\dot{H}^{\frac{3}{2}}}\leq \|B_0\|_{\dot{H}^{\frac{3}{2}}}.
\end{equation}Meanwhile, recall that $B\in C([0,T_*];L^2)$, we thus have\begin{equation}
 	\lim _{t\rightarrow 0^+}\|B(t)\|_{{H}^{\frac{3}{2}}}=\|B_0\|_{{H}^{\frac{3}{2}}},
\end{equation}namely, $B$ is strongly right continuous at $t=0$. As a result, $B\in C([0,T_*];H^{\frac{3}{2}})$.\par 
This means that the proof of \thmref{thm1} is finished.
\section{Proof of \thmref{thm2}}\label{sec4}
From the previous section, we see that the core of the proof of \thmref{thm1} is to establish the uniform bounds for the approximate solutions. In what follows, we only show how to use \eqref{extend eq} to derive the desired uniform bounds, the remainder of the proof of \thmref{thm2} can proceed in a similar way as that in the previous section and thus is omitted.\par 
\begin{prop}\label{secondproof}
	Assume that $\mu=\nu$. If $(u,B)$ is a smooth solution to \eqref{main eq} on the maximal time interval $[0,T)$, then there exists a positive time $T_*<T$ and a positive constant $C$ such that\begin{equation}
		\begin{split}
			\|u\|_{L^\infty_{T_*}({H}^{\frac{1}{2}})}+\|\nabla u\|_{L^2_{T_*}({H}^{\frac{1}{2}})}+\|B\|_{L^\infty_{T_*}({H}^{\frac{3}{2}})}+\|\nabla B\|_{L^2_{T_*}({H}^{\frac{3}{2}})}\leq C(\|u_0\|_{{H}^{\frac{1}{2}}}+\|B_0\|_{{H}^{\frac{3}{2}}}+1),
		\end{split}
	\end{equation}where $T_*$ depends on the initial data and is uniform with respect to all approximate solutions.
\end{prop}
\begin{proof}Apply the operator $\dot{\Delta}_j$ to both sides of \eqref{extend eq}, notice that\begin{equation}\label{lala2}
		\begin{split}
			&\Big(\dot{\Delta}_j(\nabla\times( B\times(\nabla\times v))),\dot{\Delta}_jv\Big)
			\\=&	\Big(\dot{\Delta}_j(B\times(\nabla\times v) )-B\times(\nabla\times\dot{\Delta}_j v) ,\nabla\times\dot{\Delta}_jv\Big)
			\\=&	\Big([\dot{\Delta}_j,B\times](\nabla\times v) ,\nabla\times\dot{\Delta}_jv\Big),
		\end{split}
	\end{equation}then following the arguments similar as \eqref{Bes}-\eqref{3}, we have that for any $t<T$ and any $p,q\in [2,\infty]$ with $1+\frac{1}{p}=\frac{1}{2}+\frac{1}{q}$,  \begin{equation}\label{cjbyfq}
		\begin{split}
			\|(\dot{\Delta}_ju,&\dot{\Delta}_jB,\dot{\Delta}_jv)\|_{L^p_t(L^2)}\lesssim \Big(\frac{1-e^{-pc2^{2j}t}}{pc2^{2j}}\Big)^{\frac{1}{p}}\|(\dot{\Delta}_ju_0,\dot{\Delta}_jB_0,\dot{\Delta}_jv_0)\|_{L^2}\\&+\Big(\frac{1-e^{-qc2^{2j}t}}{qc2^{2j}}\Big)^{\frac{1}{q}}\Big(\|\dot{\Delta}_j\Div(B\otimes B)\|_{L^2_t(L^2)}+\|\dot{\Delta}_j\Div(u\otimes u)\|_{L^2_t(L^2)}\\&+\|\dot{\Delta}_j\nabla\times(v\times  B)\|_{L^2_t(L^2)}+\|\dot{\Delta}_j\nabla\times(v\times  u)\|_{L^2_t(L^2)}\\&+\|\dot{\Delta}_j\nabla\times(v\cdot\nabla B)\|_{L^2_t(L^2)}+2^j\|[\dot{\Delta}_j,B\times](\nabla\times v)\|_{L^2_t(L^2)}\Big).
		\end{split}
	\end{equation}Multiplying both sides of \eqref{3} by $2^{j(\frac{1}{2}+\frac{2}{p})}$ and taking the $\ell^2(\mathbb{Z})$ norm, we arrive at\begin{equation}\label{muchim}
		\begin{split}
			\|(u,B&,v)\|_{\widetilde{L}_t^p(\dot{H}^{\frac{1}{2}+\frac{2}{p}})}\lesssim \Big[\sum_{j\in\mathbb{Z}}({1-e^{-pc2^{2j}t}})^{\frac{2}{p}}(2^{j\frac{1}{2}}\|(\dot{\Delta}_ju,\dot{\Delta}_jB,\dot{\Delta}_jv)\|_{L^2})^2\Big]^{\frac{1}{2}}
			\\&+\|B\otimes B\|_{L^2_t(\dot{H}^{\frac{1}{2}})}+\|u\otimes u\|_{L^2_t(\dot{H}^{\frac{1}{2}})}+\|v\times B\|_{L^2_t(\dot{H}^{\frac{1}{2}})}+\|v\times u\|_{L^2_t(\dot{H}^{\frac{1}{2}})}\\&+\|v\cdot\nabla B\|_{L^2_t(\dot{H}^{\frac{1}{2}})}+\Big\|\|2^{j\frac{1}{2}}\|[\dot{\Delta}_j,B\times](\nabla\times v)\|_{L^2}\|_{\ell^2(\mathbb{Z})}\Big\|_{L^2_t}.
		\end{split}
	\end{equation}Set $(s,r,\eta,\theta)=(\frac{1}{2},2,\frac{1}{2},\frac{1}{2})$ in \lemref{product3}, we have the following product law\begin{equation}\label{haha1}
		\|fg\|_{\dot{H}^{\frac{1}{2}}}\lesssim \|f\|_{\dot{H}^{1}}\|g\|_{\dot{H}^{1}}.
	\end{equation}Hence, applying \eqref{haha1}, the first five nonlinear terms in the right-hand side of \eqref{muchim} can be bounded by $C\|(u,B,\nabla B,v)\|_{L^4_t(\dot{H}^{1})}^2$. Thanks to \lemref{commutator} with
 $(s,r,\eta,\theta)=(\frac{1}{2},2,\frac{1}{2},\frac{3}{2})$, the commutator term can be bounded as follows\begin{equation*}
		\begin{split}
	\Big\|\|2^{j\frac{1}{2}}\|[\dot{\Delta}_j,B\times](\nabla\times v)\|_{L^2}\|_{\ell^2(\mathbb{Z})}\Big\|_{L^2_t}\lesssim
			\|B\|_{L^4_t(\dot{H}^{2})}\|v\|_{L^4_t(\dot{H}^{1})}.
		\end{split}
	\end{equation*}In view of the vector identity\begin{equation*}
		\nabla\times (\nabla\times w)+\Delta w=\nabla\Div w,
	\end{equation*}one sees that\begin{equation*}
		B=(-\Delta)^{-1}\nabla\times (\nabla\times B)=(-\Delta)^{-1}\nabla\times (u-v).
	\end{equation*}Since $(-\Delta)^{-1}\nabla\times$ is a homogeneous Fourier multiplier of degree $-1$, it follows from \lemref{prop} (iv) that\begin{equation*}\begin{split}
			\|B\|_{\dot{H}^{s+1}}\lesssim  \|u\|_{\dot{H}^{s}}+\|v\|_{\dot{H}^{s}}.
		\end{split}
	\end{equation*}Therefore, all the nonlinear terms are controlled by $C\|(u,B,v)\|_{L^4_t(\dot{H}^{1})}^2$. Taking $p=4$ in \eqref{muchim} gives us that there exists a positive constant $C_3$ such that for any $t<T$,\begin{equation}\label{cool}
		\begin{split}
			\|(u,B,v)\|_{{L}_t^4(\dot{H}^{1})}\leq &C_3\Big( \Big[\sum_{j\in\mathbb{Z}}({1-e^{-4c2^{2j}t}})^{\frac{1}{2}}(2^{j\frac{1}{2}}\|(\dot{\Delta}_ju,\dot{\Delta}_jB,\dot{\Delta}_jv)\|_{L^2})^2\Big]^{\frac{1}{2}}\\&+\|(u,B,v)\|_{L^4_t(\dot{H}^{1})}^2\Big).
		\end{split}
	\end{equation}Recall the fact that for any given $\delta>0$, there exists a positive time $T_*$ such that\begin{equation*}
		\Big[\sum_{j\in\mathbb{Z}}({1-e^{-4c2^{2j}T_*}})^{\frac{1}{2}}(2^{j\frac{1}{2}}\|(\dot{\Delta}_ju,\dot{\Delta}_jB,\dot{\Delta}_jv)\|_{L^2})^2\Big]^{\frac{1}{2}}<\delta.
	\end{equation*}Define\begin{equation*}
		T_{**}=\sup\Big\{t\in[0,\min\{T,T_*\}),\,\,\|(u,B,v)\|_{{L}_t^4(\dot{H}^{1})}\leq2C_3\delta\Big\},
	\end{equation*}where $\delta$ will be determined later on. Hence, it follows from \eqref{cool} that for all $t\in[0,T_{**})$,\begin{equation}\begin{split}
			\|(u,B,v)\|_{{L}_t^4(\dot{H}^{1})}&\leq C_3\delta(1+4C_3^2\delta).
		\end{split}
	\end{equation}Taking $\delta$ small enough such that $4C_3^2\delta<\frac{1}{2}$ yields that for all $t\in[0,T_{**})$,\begin{equation}\begin{split}
			\|(u,B,v)\|_{{L}_t^4(\dot{H}^{1})}&< \frac{3}{2}C_3\delta.
		\end{split}
	\end{equation}Then by the continuity of $\|(u,B,v)\|_{{L}_t^4(\dot{H}^{1})}$, we have $T_{**}=\min\{T,T_*\}$.\par We assert that $T_*<T$. Indeed, if $T\leq T_*<\infty$, we infer from \eqref{muchim} with $p=2$ that\begin{equation*}
		\|(u,B,v)\|_{{L}_T^2(\dot{H}^{\frac{3}{2}})}\lesssim \|(u_0,B_0,v_0)\|_{\dot{H}^{\frac{1}{2}}}+\|(u,B,v)\|_{{L}_T^4(\dot{H}^{1})}^2<\infty.
	\end{equation*}By $\dot{H}^{\frac{3}{2}}(\mathbb{R}^3)\hookrightarrow \text{BMO}(\mathbb{R}^3)$, the above inequality contradicts with the blow-up criterion \eqref{blowup}. Consequently, $T_*<T$. Thus, combining \eqref{nice0} and \eqref{muchim} leads to\begin{equation}
		\begin{split}
			&\|u\|_{L^\infty_{T_*}({H}^{\frac{1}{2}})}+\|\nabla u\|_{L^2_{T_*}({H}^{\frac{1}{2}})}+\|B\|_{L^\infty_{T_*}({H}^{\frac{3}{2}})}+\|\nabla B\|_{L^2_{T_*}({H}^{\frac{3}{2}})}
			\\
			\lesssim& \|u\|_{L^\infty_{T_*}(L^2)}+\|B\|_{L^\infty_{T_*}(L^2)}+\|\nabla u\|_{L^2_{T_*}(L^2)}+\|\nabla B\|_{L^2_{T_*}(L^2)}\\&+\|u\|_{\widetilde{L}^\infty_{T_*}(\dot{H}^{\frac{1}{2}})}+\|v\|_{\widetilde{L}^\infty_{T_*}(\dot{H}^{\frac{1}{2}})}+\|u\|_{\widetilde{L}^2_{T_*}(\dot{H}^{\frac{3}{2}})}+\|v\|_{\widetilde{L}^2_{T_*}(\dot{H}^{\frac{3}{2}})}\\\lesssim& \|u_0\|_{{H}^{\frac{1}{2}}}+\|B_0\|_{{H}^{\frac{3}{2}}}+\|(u,B,v)\|_{{L}_{T_*}^4(\dot{H}^{1})}^2\\
			\lesssim& \|u_0\|_{{H}^{\frac{1}{2}}}+\|B_0\|_{{H}^{\frac{3}{2}}}+\delta^2.
		\end{split}
	\end{equation}This completes the proof of \propref{secondproof}.
\end{proof}
\section{Proof of \thmref{thm3}}\label{section3}
This section is devoted to \thmref{thm3}. In the first subsection, we extend the local solutions constructed in \thmref{thm1} globally in time provided the initial data is small enough. The optimal time-decay rates of solutions will be derived in the second subsection.  
\subsection{Global existence for small initial data}
\quad\, We denote by $T^*$ the largest time such that there holds \eqref{local} in \thmref{thm1}. It suffices to prove that $T^*=\infty$ under the assumption \eqref{smallcon}. \par 
From \eqref{gg2}, we have that for all $t\in [0,T^*)$, \begin{equation}\label{globalap}
	\begin{split}
	&\|u\|_{L^\infty_t({H}^{\frac{1}{2}+\sigma})}+\|\nabla u\|_{L^2_t({H}^{\frac{1}{2}+\sigma})}+\|B\|_{L^\infty_T({H}^{\frac{3}{2}})}+\|\nabla B\|_{L^2_t({H}^{\frac{3}{2}})}\\&\lesssim \|u_0\|_{{H}^{\frac{1}{2}+\sigma}}+\|B_0\|_{{H}^{\frac{3}{2}}}+\|u\cdot\nabla u\|_{L^2_t(\dot{H}^{\sigma-\frac{1}{2}})}+\|B\cdot\nabla B\|_{L^2_t(\dot{H}^{\sigma-\frac{1}{2}})}\\&\quad+\|\nabla\times(u\times B)\|_{L^2_t(\dot{H}^{\frac{1}{2}})}+\Big\|\|2^{j\frac{3}{2}}\|[\dot{\Delta}_j,B\times](\nabla\times B)\|_{L^2}\|_{\ell^2(\mathbb{Z})}\Big\|_{L^2_t}.
	\end{split}
\end{equation}Thanks to the product laws \eqref{ppr1} and \eqref{ppr2}, we derive\begin{equation}\label{hope111}
\begin{split}
	\|u\cdot\nabla u\|_{L^2_T(\dot{H}^{\sigma-\frac{1}{2}})}&\lesssim \|u\otimes u\|_{L^2_T(\dot{H}^{\frac{1}{2}+\sigma})}\\
	&\lesssim \|u\|_{L^\infty_T(\dot{H}^{\frac{1}{2}+\frac{\sigma}{2}})}\|u\|_{L^2_T(\dot{H}^{\frac{3}{2}+\frac{\sigma}{2}})}\\
	&\lesssim \|u\|_{L^\infty_T(\dot{H}^{\frac{1}{2}+\frac{\sigma}{2}})}\|\nabla u\|_{L^2_T(\dot{H}^{\frac{1}{2}+\frac{\sigma}{2}})}\\
	&\lesssim \|u\|_{L^\infty_T({H}^{\frac{1}{2}+\sigma})}\|\nabla u\|_{L^2_T({H}^{\frac{1}{2}+\sigma})}.\end{split}
\end{equation}and\begin{equation}
\begin{split}
	\|B\cdot\nabla B\|_{L^2_T(\dot{H}^{\sigma-\frac{1}{2}})}&\lesssim\|B\otimes B\|_{L^2_T(\dot{H}^{\frac{1}{2}+\sigma})}\\
	&\lesssim \|B\|_{L^\infty_T(\dot{H}^{\frac{1}{2}+\frac{\sigma}{2}})}\|B\|_{L_T^{2}(\dot{H}^{\frac{3}{2}+\frac{\sigma}{2}})}\\
	&\lesssim \|B\|_{L^\infty_T(\dot{H}^{\frac{1}{2}+\frac{\sigma}{2}})}\|\nabla B\|_{L_T^{2}(\dot{H}^{\frac{1}{2}+\frac{\sigma}{2}})}\\
	&\lesssim \|B\|_{L^\infty_T({H}^{\frac{3}{2}})}\|\nabla B\|_{L_T^{2}({H}^{\frac{3}{2}})}.\end{split}
\end{equation}and
\begin{equation}
\begin{split}
	&\|\nabla\times(u\times B)\|_{L^2_T(\dot{H}^{\frac{1}{2}})}
	\lesssim \|u\times B\|_{L^2_T(\dot{H}^{\frac{3}{2}})}\\
	&\lesssim  \|u\|_{L^\infty_T(\dot{H}^{\frac{1}{2}+\frac{\sigma}{2}})}\|B\|_{L^2_T(\dot{H}^{\frac{5}{2}-\frac{\sigma}{2}})}+\|u\|_{L^2_T(\dot{H}^{\frac{3}{2}+\frac{\sigma}{2}})}\|B\|_{L^\infty_T(\dot{H}^{\frac{3}{2}-\frac{\sigma}{2}})}\\
	&\lesssim \|u\|_{L^\infty_T(\dot{H}^{\frac{1}{2}+\frac{\sigma}{2}})}\|\nabla B\|_{L^2_T(\dot{H}^{\frac{3}{2}-\frac{\sigma}{2}})}+\|\nabla u\|_{L^2_T(\dot{H}^{\frac{1}{2}+\frac{\sigma}{2}})}\|B\|_{L^\infty_T(\dot{H}^{\frac{3}{2}-\frac{\sigma}{2}})}\\
	&\lesssim \|u\|_{L^\infty_T({H}^{\frac{1}{2}+\sigma})}\|\nabla B\|_{L^2_T({H}^{\frac{3}{2}})}+\|\nabla u\|_{L^2_T({H}^{\frac{1}{2}+\sigma})}\|B\|_{L^\infty_T({H}^{\frac{3}{2}})}.
\end{split}
\end{equation}
Tackling the commutator term as \eqref{target2} and using the interpolation, we can get
\begin{equation}\label{commueses}
\begin{split}
	&\Big\|\|2^{j\frac{3}{2}}\|[\dot{\Delta}_j,B\times](\nabla\times B)\|_{L^2}\|_{\ell^2(\mathbb{Z})}\Big\|_{L^2_t}\\&\lesssim  \|B\|_{L^4_t(\dot{H}^{2})}^2
	\lesssim \Big(\|B\|_{L^\infty_t(\dot{H}^{\frac{3}{2}})}^{\frac{1}{2}}\|\nabla B\|_{L^2_t(\dot{H}^{\frac{3}{2}})}^{\frac{1}{2}}\Big)^2\\
	&\lesssim \|B\|_{L^\infty_t({H}^{\frac{3}{2}})}\|\nabla B\|_{L^2_t({H}^{\frac{3}{2}})}.
\end{split}
\end{equation}Denote\begin{equation*}
\mathcal{E}(t):=\|u\|_{L^\infty_t({H}^{\frac{1}{2}+\sigma})}+\|\nabla u\|_{L^2_t({H}^{\frac{1}{2}+\sigma})}+\|B\|_{L^\infty_t({H}^{\frac{3}{2}})}+\|\nabla B\|_{L^2_t({H}^{\frac{3}{2}})}.
\end{equation*}Plugging \eqref{hope111}-\eqref{commueses} into \eqref{globalap}, we can deduce that there exists a constant $C_4>1$ such that for all $t\in [0,T^*)$, \begin{equation}\label{es123}
\begin{split}
\mathcal{E}(t)\leq C_4(\mathcal{E}(0)+\mathcal{E}^2(t)).
\end{split}
\end{equation}
Define\begin{equation}
T^{**}:=\sup\Big\{t\in [0,T^*):\,\mathcal{E}(t)\leq 2C_4\mathcal{E}(0)\Big\}.
\end{equation}It follows from \eqref{es123} that for all $t\in[0,T^{**})$ that\begin{equation}
\mathcal{E}(t)\leq C_4(\mathcal{E}(0)+4C_4^2\mathcal{E}^2(0)).
\end{equation}Since $\mathcal{E}(0)\leq c_0$ where $c_0$ is given in \eqref{smallcon}, taking $c_0$ small enough such that $4C_4^2c_0<\frac{1}{2}$, we can get\begin{equation}
\mathcal{E}(t)<\frac{3}{2}C_4\mathcal{E}(0)\quad\text{for}\quad t\in[0,T^{**}). 
\end{equation}Then, by the continuity of $\mathcal{E}(t)$, we conclude that $T^{**}=T^*$.\par
We assert that $T^*=\infty$. Becasue $\mathcal{E}(t)$ is uniformly bounded on $[0,T^*)$ (and notice that  $\|B\|_{\widetilde{L}^{\infty}_{T^*}(\dot{H}^{\frac{3}{2}})}<\infty$), it follows from \thmref{thm1} and \eqref{fd2} that for any $t_0\in [0,T^*)$, we can solve \eqref{main eq} on $[t_0,t_0+T^\prime]$ with some fixed constant $T^\prime$. If $T^*<+\infty$, choosing $t_0>T^*-T^\prime$ means that the solution $(u,B)$ can be continued beyond $T^*$, which contradicts with the definition of $T^*$. Thus, $T^*=+\infty$. This completes the proof of the global well-posedness part of \thmref{thm3}.
\par 

\subsection{Optimal time-decay rates of solutions}
\quad\, In this subsection, we are going to derive the optimal time-decay rates of the global solutions constructed in the previous subsection. Recall that $(u(t),B(t))\in H^\infty$ after $t>0$. By performing a standard energy argument, one easily have\begin{equation}\label{ghgjd}
	\begin{split}
		&\frac{d}{dt}(\|\Lambda^su\|_{L^2}^2+\|\Lambda^sB\|_{L^2}^2)+\|\Lambda^{s+1}u\|_{L^2}^2+\|\Lambda^{s+1}B\|_{L^2}^2\\&\lesssim (\|u\cdot \nabla u\|_{\dot{H}^{s-1}}+\|B\cdot \nabla B\|_{\dot{H}^{s-1}})\|\Lambda^{s+1}u\|_{L^2}\\&\quad+(\|\nabla\times(u\times B)\|_{\dot{H}^{s-1}}+\|[\Lambda^s,B\times](\nabla\times B)\|_{L^2})\|\Lambda^{s+1}B\|_{L^2}.
	\end{split}
\end{equation}Set $(s,r,\eta,\theta)=(l,2,1,1)$ in \lemref{product3}, the following product law holds:\begin{equation}\label{plaw2}
\|fg\|_{\dot{H}^l}\lesssim \|f\|_{\dot{H}^{\frac{1}{2}}}\|g\|_{\dot{H}^{l+1}}+\|g\|_{\dot{H}^{\frac{1}{2}}}\|f\|_{\dot{H}^{l+1}}\quad\text{for}\quad  l>-\frac{3}{2}.
\end{equation}Hence, using \eqref{plaw2} gives\begin{equation}\label{henkl1}
\begin{split}
\|\nabla\times(u\times B)\|_{\dot{H}^{s-1}}&\lesssim \|u\times B\|_{\dot{H}^{s}}\\&\lesssim \|u\|_{\dot{H}^{\frac{1}{2}}}\|B\|_{\dot{H}^{s+1}}+\|B\|_{\dot{H}^{\frac{1}{2}}}\|u\|_{\dot{H}^{s+1}}\\&\lesssim \|u\|_{{H}^{\frac{1}{2}+\sigma}}\|B\|_{\dot{H}^{s+1}}+\|B\|_{{H}^{\frac{3}{2}}}\|u\|_{\dot{H}^{s+1}}.
\end{split}
\end{equation}Similarly,\begin{equation}
\begin{split}
	\|u\cdot \nabla u\|_{\dot{H}^{s-1}}&\lesssim \|u\|_{{H}^{\frac{1}{2}+\sigma}}\|u\|_{\dot{H}^{s+1}},\\\|B\cdot \nabla B\|_{\dot{H}^{s-1}}&\lesssim \|B\|_{{H}^{\frac{3}{2}}}\|B\|_{\dot{H}^{s+1}}.
\end{split}
\end{equation}Taking advantage of \lemref{commues3} and \lemref{GN}, we have
\begin{equation}\label{henkl2}
	\begin{split}
\|[\Lambda^s,B\times](\nabla\times B)\|_{L^2}&\lesssim \|\Lambda^s B\|_{L^6}\|\nabla\times B\|_{L^3}+\|\Lambda^{s-1} (\nabla\times B)\|_{L^6}\|\nabla B\|_{L^3}\\&\lesssim \|B\|_{\dot{H}^{s+1}}\|B\|_{H^{\frac{3}{2}}}.
	\end{split}
\end{equation}Then, inserting \eqref{henkl1}-\eqref{henkl2} into \eqref{ghgjd} yields\begin{equation}
\begin{split}
	&\frac{d}{dt}(\|\Lambda^su\|_{L^2}^2+\|\Lambda^sB\|_{L^2}^2)+\|\Lambda^{s+1}u\|_{L^2}^2+\|\Lambda^{s+1}B\|_{L^2}^2\\&\lesssim (\|u\|_{{H}^{\frac{1}{2}+\sigma}}+\|B\|_{{H}^{\frac{3}{2}}})(\|\Lambda^{s+1}u\|_{L^2}^2+\|\Lambda^{s+1}B\|_{L^2}^2)\\&\lesssim \mathcal{E}(0)(\|\Lambda^{s+1}u\|_{L^2}^2+\|\Lambda^{s+1}B\|_{L^2}^2).
\end{split}
\end{equation}Since $\mathcal{E}(0)\ll 1$, we conclude that \begin{equation}\label{xiaoyl}
\begin{split}
	\frac{d}{dt}(\|\Lambda^su\|_{L^2}^2+\|\Lambda^sB\|_{L^2}^2)+c_1(\|\Lambda^{s+1}u\|_{L^2}^2+\|\Lambda^{s+1}B\|_{L^2}^2)\leq 0
\end{split}
\end{equation}for some constant $c_1>0$. Assume for the moment that there holds\begin{equation}\label{shhengde}
\|u(t)\|_{\dot{B}^{-\gamma}_{2,\infty}}+\|B(t)\|_{\dot{B}^{-\gamma}_{2,\infty}}\leq C_0\quad\text{for all}\quad t\geq 0
\end{equation}with some constant $C_0>0$, which will be proved in \propref{neevolu} below later. If $s>0$, it follows from the interpolation (see \eqref{inter2} in \lemref{prop} (iii)) that\begin{equation}
\|\Lambda^s(u,B)\|_{L^2}\lesssim \Big(\|(u,B)\|_{\dot{B}^{-\gamma}_{2,\infty}}\Big)^{\frac{1}{s+\gamma+1}}\Big(\|\Lambda^{s+1}(u,B)\|_{L^2}\Big)^{1-\frac{1}{s+\gamma+1}}.
\end{equation}By means of \eqref{shhengde}, we get\begin{equation}
\begin{split}
	\|\Lambda^{s+1}(u,B)\|_{L^2}\geq C\Big(\|\Lambda^s(u,B)\|_{L^2}\Big)^{1+\frac{1}{s+\gamma}}.
\end{split}
\end{equation}Thus, thanks to \eqref{xiaoyl}, there exists a constant $c_2>0$ such that the following Lyapunov-type inequality holds\begin{equation}
\begin{split}
	\frac{d}{dt}(\|\Lambda^su\|_{L^2}^2+\|\Lambda^sB\|_{L^2}^2)+c_2\Big(\|\Lambda^{s}u\|_{L^2}^2+\|\Lambda^{s}B\|_{L^2}^2\Big)^{1+\frac{1}{s+\gamma}}\leq 0.
\end{split}
\end{equation}Solving this inequality directly gives that for all $t\geq 1$,\begin{equation}
\begin{split}
\|\Lambda^su(t)\|_{L^2}^2+\|\Lambda^sB(t)\|_{L^2}^2&\leq \Big[\Big(\|\Lambda^su(1)\|_{L^2}^2+\|\Lambda^sB(1)\|_{L^2}^2\Big)^{-\frac{1}{s+\gamma}}+\frac{c_2}{s+\gamma}(t-1)\Big]^{-(s+\gamma)}\\
&\lesssim (1+t)^{-(s+\gamma)},
\end{split}
\end{equation}which implies \eqref{decay} with $s>0$. On the other hand, if $\gamma>0$, \eqref{decay} with $s=0$ can be obtained by using the interpolation\begin{equation}\begin{split}
\|(u,B)(t)\|_{L^2}\lesssim \|(u,B)(t)\|_{\dot{B}^{-\gamma}_{2,\infty}}^{\frac{1}{1+\gamma}}\|(u,B)(t)\|_{\dot{H}^1}^{\frac{\gamma}{1+\gamma}}\lesssim(1+t)^{-\frac{\gamma}{2}}.
\end{split}
\end{equation}If $\gamma=0$, \eqref{decay} with $s=0$ certainly holds. This means that the proof of \eqref{decay} is finished.\par 
At the last step, we should get back to prove \eqref{shhengde}, namely, the evolution of negative Besov norms of solutions. Actually, we have the following statement. 
\begin{prop}\label{neevolu}
	Let $\gamma\in[0,\frac{5}{2}]$ and $(u_0,B_0)\in \dot{B}^{-\gamma}_{2,\infty}(\mathbb{R}^3)$, it holds that for all $t\geq 0$,\begin{equation}\label{nbound}
		\|u(t)\|_{\dot{B}^{-\gamma}_{2,\infty}}+\|B(t)\|_{\dot{B}^{-\gamma}_{2,\infty}}\leq C_0,
	\end{equation}where the constant $C_0>0$ depends on the norms of the initial data.
\end{prop}
\begin{proof}
	We first consider the case that $\gamma\in[0,\frac{3}{2})$. Apply the operator $\dot{\Delta}_j$ to \eqref{main eq}, by means of a standard energy argument, the Berstein inequality and Young's inequality, we have\begin{equation}
		\begin{split}
			&\frac{1}{2}\frac{d}{dt}\|\dot{\Delta}_ju\|_{L^2}^2+\|\dot{\Delta}_j\nabla u\|_{L^2}^2\leq \Big(\|\dot{\Delta}_j(u\cdot\nabla u)\|_{L^2}+\|\dot{\Delta}_j(B\cdot\nabla B)\|_{L^2}\Big)\|\dot{\Delta}_ju\|_{L^2}\\&\leq C2^{-2j}\Big(\|\dot{\Delta}_j(u\cdot\nabla u)\|_{L^2}^2+\|\dot{\Delta}_j(B\cdot\nabla B)\|_{L^2}^2\Big)+\frac{1}{2}\|\nabla\dot{\Delta}_ju\|_{L^2}^2,\\
			&\frac{1}{2}\frac{d}{dt}\|\dot{\Delta}_jB\|_{L^2}^2+\|\dot{\Delta}_j\nabla B\|_{L^2}^2\leq
			\|\dot{\Delta}_j\nabla\times(u\times B)\|_{L^2}\|\dot{\Delta}_jB\|_{L^2} \\
			&\qquad+\|[\dot{\Delta}_j,B\times](\nabla \times B)\|_{L^2}\|\dot{\Delta}_j(\nabla \times B)\|_{L^2}\\&\leq C\Big(2^{-2j}\|\dot{\Delta}_j\nabla\times(u\times B)\|_{L^2}^2+\|[\dot{\Delta}_j,B\times](\nabla \times B)\|_{L^2}^2\Big)+\frac{1}{2}\|\nabla\dot{\Delta}_jB\|_{L^2}^2,
		\end{split}
	\end{equation}whence, after time integration, multiplying both sides by $2^{-2j\gamma}$ and taking the $\ell^\infty(\mathbb{Z})$ norm yields\begin{equation}\label{a0}
		\begin{split}
			&\|u(t)\|_{\dot{B}^{-\gamma}_{2,\infty}}^2+\|B(t)\|_{\dot{B}^{-\gamma}_{2,\infty}}^2\\&\leq \|u_0\|_{\dot{B}^{-\gamma}_{2,\infty}}^2+\|B_0\|_{\dot{B}^{-\gamma}_{2,\infty}}^2+C\int_{0}^{t}\Big(\|u\cdot\nabla u\|_{\dot{B}^{-\gamma-1}_{2,\infty}}^2+\|B\cdot\nabla B\|_{\dot{B}^{-\gamma-1}_{2,\infty}}^2\\&\qquad+\|\nabla\times(u\times B)\|_{\dot{B}^{-\gamma-1}_{2,\infty}}^2+\|2^{-j\gamma}\|[\dot{\Delta}_j,B\times](\nabla \times B)\|_{L^2}\|_{\ell^\infty(\mathbb{Z})}^2\Big)d\tau,
		\end{split}
	\end{equation}Set $(s,r,\eta,\theta)$=$(-\gamma,\infty,\gamma+\frac{3}{2},\gamma+\frac{3}{2})$ in \lemref{product3}, the following product law holds:\begin{equation}\label{jiayou}
		\|fg\|_{\dot{B}^{-\gamma}_{2,\infty}}\lesssim \|f\|_{\dot{B}^{-\gamma}_{2,\infty}}\|g\|_{\dot{B}^{\frac{3}{2}}_{2,\infty}}+\|g\|_{\dot{B}^{-\gamma}_{2,\infty}}\|f\|_{\dot{B}^{\frac{3}{2}}_{2,\infty}}.
	\end{equation}Hence, resorting to \eqref{jiayou} gives\begin{equation}\label{a3}
		\begin{split}
			\|\nabla\times(u\times B)\|_{\dot{B}^{-\gamma-1}_{2,\infty}}&\lesssim \|u\times B\|_{\dot{B}^{-\gamma}_{2,\infty}}
			\\&\lesssim \|u\|_{\dot{B}^{-\gamma}_{2,\infty}}\|\nabla B\|_{\dot{B}^{\frac{1}{2}}_{2,\infty}}+\|B\|_{\dot{B}^{-\gamma}_{2,\infty}}\|\nabla u\|_{\dot{B}^{\frac{1}{2}}_{2,\infty}}\\
			&\lesssim \|u\|_{\dot{B}^{-\gamma}_{2,\infty}}\|\nabla B\|_{H^{\frac{3}{2}}}+\|B\|_{\dot{B}^{-\gamma}_{2,\infty}}\|\nabla u\|_{H^{\frac{1}{2}+\sigma}}.
		\end{split}
	\end{equation}Similarly,\begin{equation}\label{nels1}
		\begin{split}
			\|u\cdot\nabla u\|_{\dot{B}^{-\gamma-1}_{2,\infty}}
			&\lesssim \|u\|_{\dot{B}^{-\gamma}_{2,\infty}}\|\nabla u\|_{H^{\frac{1}{2}+\sigma}},\\
			\|B\cdot\nabla B\|_{\dot{B}^{-\gamma-1}_{2,\infty}}&\lesssim \|B\|_{\dot{B}^{-\gamma}_{2,\infty}}\|\nabla B\|_{H^{\frac{3}{2}}}.
		\end{split}
	\end{equation}Applying \lemref{commutator} with $(s,r,\eta,\theta)=(-\gamma,\infty,\frac{1}{2},\frac{1}{2})$ and interpolation yields\begin{equation}\label{a5}
		\begin{split}
			&\|2^{-j\gamma}\|[\dot{\Delta}_j,B\times](\nabla \times B)\|_{L^2}\|_{\ell^\infty(\mathbb{Z})}\\\lesssim& \| B\|_{\dot{B}^{2}_{2,\infty}}\| \nabla\times B\|_{\dot{B}^{-\gamma-\frac{1}{2}}_{2,\infty}}+\| \nabla\times B\|_{\dot{B}^{1}_{2,\infty}}\| B\|_{\dot{B}^{-\gamma+\frac{1}{2}}_{2,\infty}}\\
			\lesssim& \Big(\| B\|_{\dot{B}^{-\gamma}_{2,\infty}}^{\frac{1}{2\gamma+5}}\|\nabla B\|_{\dot{B}^{\frac{3}{2}}_{2,\infty}}^{\frac{2\gamma+4}{2\gamma+5}}\Big)\Big(\| B\|_{\dot{B}^{-\gamma}_{2,\infty}}^{\frac{2\gamma+4}{2\gamma+5}}\|\nabla B\|_{\dot{B}^{\frac{3}{2}}_{2,\infty}}^{\frac{1}{2\gamma+5}}\Big)\\
			\lesssim& \|B\|_{\dot{B}^{-\gamma}_{2,\infty}}\|\nabla B\|_{H^{\frac{3}{2}}}.
		\end{split}
	\end{equation}
	Therefore, inserting \eqref{a3}-\eqref{a5} into \eqref{a0}, we deduce that\begin{equation}\label{a111}
		\begin{split}
			\|u(t)\|_{\dot{B}^{-\gamma}_{2,\infty}}^2&+\|B(t)\|_{\dot{B}^{-\gamma}_{2,\infty}}^2\leq \|u_0\|_{\dot{B}^{-\gamma}_{2,\infty}}^2+\|B_0\|_{\dot{B}^{-\gamma}_{2,\infty}}^2\\&+C\int_{0}^{t}\Big(\|\nabla u\|_{H^{\frac{1}{2}+\sigma}}^2+\|\nabla B\|_{H^{\frac{3}{2}}}^2\Big)\Big(\|u\|_{\dot{B}^{-\gamma}_{2,\infty}}^2+\|B\|_{\dot{B}^{-\gamma}_{2,\infty}}^2\Big)d\tau. 
		\end{split}
	\end{equation}Applying the Gronwall's inequality to \eqref{a111} yields \eqref{nbound} with $\gamma\in [0,\frac{3}{2})$.\par 
	Next, we consider the case that $\gamma\in[\frac{3}{2},\frac{5}{2}]$. By performing a routine procedure, one can reach\begin{equation}\label{xingj}
		\begin{split}
			&\|u(t)\|_{\dot{B}^{-\gamma}_{2,\infty}}^2+\|B(t)\|_{\dot{B}^{-\gamma}_{2,\infty}}^2\\&\leq \|u_0\|_{\dot{B}^{-\gamma}_{2,\infty}}^2+\|B_0\|_{\dot{B}^{-\gamma}_{2,\infty}}^2+C\int_{0}^{t}\Big(\|u\cdot\nabla u\|_{\dot{B}^{-\gamma}_{2,\infty}}+\|B\cdot\nabla B\|_{\dot{B}^{-\gamma}_{2,\infty}}\\&\quad+\|\nabla\times(u\times B)\|_{\dot{B}^{-\gamma}_{2,\infty}}+\|\nabla\times((\nabla\times B)\times B)\|_{\dot{B}^{-\gamma}_{2,\infty}}\Big)\Big(\|u\|_{\dot{B}^{-\gamma}_{2,\infty}}+\|B\|_{\dot{B}^{-\gamma}_{2,\infty}}\Big)d\tau.
		\end{split}
	\end{equation}By means of embeddings, H\"{o}lder's inequality and \lemref{GN}, we have
	\begin{equation}
		\begin{split}
					\|\nabla\times(u\times B)\|_{\dot{B}^{-\gamma}_{2,\infty}}&\lesssim \|u\times B\|_{\dot{B}^{0}_{\frac{6}{2\gamma+1},\infty}}\lesssim \|u\times B\|_{L^{\frac{6}{2\gamma+1}}}\\&\lesssim \|u\|_{L^{\frac{12}{2\gamma+1}}}\|B\|_{L^{\frac{12}{2\gamma+1}}}\\&\lesssim \|u\|_{\dot{H}^{\frac{5-2\gamma}{4}}}\|B\|_{\dot{H}^{\frac{5-2\gamma}{4}}}.
		\end{split}
	\end{equation}Similarly,\begin{equation}
		\begin{split}
			\|u\cdot\nabla u\|_{\dot{B}^{-\gamma}_{2,\infty}}&\lesssim \|u\|_{\dot{H}^{\frac{5-2\gamma}{4}}}^2,\\
			\|B\cdot\nabla B\|_{\dot{B}^{-\gamma}_{2,\infty}}&\lesssim \|B\|_{\dot{H}^{\frac{5-2\gamma}{4}}}^2,\\
			\|\nabla\times((\nabla\times B)\times B)\|_{\dot{B}^{-\gamma}_{2,\infty}}&\lesssim \|B\|_{\dot{H}^{\frac{9-2\gamma}{4}}}\|B\|_{\dot{H}^{\frac{5-2\gamma}{4}}}.
		\end{split}
	\end{equation}Observe that $(u_0,B_0)\in \dot{B}^{-\frac{5}{4}}_{2,\infty}$ because of the embedding $\dot{B}^{-\gamma}_{2,\infty}\cap L^2\hookrightarrow \dot{B}^{-\gamma^\prime}_{2,\infty}$ for any $\gamma^\prime\in [0,\gamma]$. Since the evolution of $\dot{B}^{-\frac{5}{4}}_{2,\infty}$ norm has been derived in \eqref{a111} and $0\leq\frac{5-2\gamma}{4}\leq\frac{1}{2}\leq\frac{9-2\gamma}{4}\leq\frac{3}{2}$, we can obtain from what we have proved  for \eqref{decay} with $\gamma=\frac{5}{4}$ that for all $t\geq 0$,\begin{equation}\label{hhaodff}
	\begin{split}
		\|(u,B)(t)\|_{\dot{H}^{\frac{5-2\gamma}{4}}}\lesssim (1+t)^{-\frac{5-\gamma}{4}},\quad \|B(t)\|_{\dot{H}^{\frac{9-2\gamma}{4}}}\lesssim (1+t)^{-\frac{7-\gamma}{4}}.
	\end{split}
\end{equation}Therefore, denote\begin{equation*}
\mathcal{F}(t):=\|u(t)\|_{\dot{B}^{-\gamma}_{2,\infty}}^2+\|B(t)\|_{\dot{B}^{-\gamma}_{2,\infty}}^2,
\end{equation*}then it follows from \eqref{xingj} that\begin{equation}
		\begin{split}
			\mathcal{F}(t)&\leq \|u_0\|_{\dot{B}^{-\gamma}_{2,\infty}}^2+\|B_0\|_{\dot{B}^{-\gamma}_{2,\infty}}^2+C\int_{0}^{t}(1+\tau)^{-\frac{5-\gamma}{2}}d\tau \sup_{0\leq\tau\leq t}\sqrt{\mathcal{F}(\tau)}\\
			&\lesssim 1+\sup_{0\leq\tau\leq t}\sqrt{\mathcal{F}(\tau)},
		\end{split}
	\end{equation}which implies \eqref{nbound} with $\gamma\in[\frac{3}{2},\frac{5}{2}]$. This completes the proof of \propref{neevolu}.
\end{proof}
\noindent \textbf{Acknowledgments}  \quad \\
The author is grateful to the referee for the helpful comments and suggestions.
\\[10pt] 
\noindent \textbf{Declaration}  \quad \\
\textbf{Conflict of interest}  The author states that there is no conflict of interest.
\bibliographystyle{abbrv} 
\bibliography{bioo}
\textsc{School of Mathematics and Computational Science, Wuyi University, Jiangmen 529020, P.R.China}\par 
E-mail address: zhangshunhang@wyu.edu.cn
\end{document}